\documentclass[reqno,12pt]{amsart}
\usepackage{amssymb, amsthm, amsmath}
\usepackage{version}
\usepackage{bm}
\usepackage{color}
\usepackage{mathrsfs}
\newtheorem{definition}{\sc Definition}[section]
\newtheorem{theorem}[definition]{\sc Theorem}
\newtheorem{lemma}[definition]{\sc Lemma}

\newtheorem{proposition}[definition]{\sc Proposition}
\newtheorem{corollary}[definition]{\sc Corollary}
\newtheorem{remark}[definition]{\sc Remark}

\newcommand{\normtre}[1]{{\left\vert\kern-0.25ex\left\vert\kern-0.25ex\left\vert #1 \right\vert\kern-0.25ex\right\vert\kern-0.25ex\right\vert}}

\makeatletter

\@addtoreset{equation}{section}

\renewcommand{\d}{\text{\rm d}}
\newcommand{\vep}{\varepsilon}
\newcommand{\R}{\mathbb{R}}
\newcommand{\N}{\mathbb{N}}
\renewcommand{\dim}{N}

\newcommand{\e}{\text{\rm e}}
\newcommand{\sgn}{{\rm sgn}}
\newcommand{\lam}{\lambda}
\newcommand{\ddelta}{\delta}

\newcommand{\Lap}{\Delta_H}

\newcommand{\Rd}{\mathbb R^\dim}
\newcommand{\dH}{\nabla_\xi H}
\newcommand{\B}{{\sf B}}
\newcommand{\BR}{B_R}
\newcommand{\FBR}{\B_R}
\newcommand{\Br}{B_{R/2}}
\newcommand{\Sl}{S}
\newcommand{\phir}{\phi_r}
\newcommand{\psir}{\psi_r}
\newcommand{\Wpr}{W^{1,p}(\BR)}
\newcommand{\tI}{\hat{t}_1}
\newcommand{\tJ}{\hat{t}_2}
\newcommand{\Tr}{T_r}
\newcommand{\T}{T}

\def\Vec#1{\mbox{\boldmath $#1$}}

\allowdisplaybreaks[4]
\newcommand{\prf}[1]{}  % without proof

\begin{document}

%\begin{frontmatter}

\title[General framework for nonlinear diffusion equations]
{General framework to construct local-energy solutions of nonlinear diffusion equations \\for growing initial data\prf{ \tt [extended version]}}

\author{Goro Akagi}
\address{Mathematical Institute and Graduate School of Science, Tohoku University, 6-3 Aoba, Aramaki, Aoba-ku, Sendai 980-8578, Japan.}
\email{goro.akagi@tohoku.ac.jp}

\author{Kazuhiro Ishige}
\address{Graduate School of Mathematical Sciences, The University of Tokyo, 3-8-1 Komaba, Meguro-ku, Tokyo 153-8914, Japan.}
\email{ishige@ms.u-tokyo.ac.jp}

\author{Ryuichi Sato}
\address{Department of Applied Mathematics, Faculty of Science, Fukuoka University, Fukuoka 814-0180, Japan.}
%\address{Graduate School of Mathematical Sciences, The University of Tokyo, 3-8-1 Komaba, Meguro-ku, Tokyo 153-8914, Japan.}
%\address{Mathematical Institute and Graduate School of Science, Tohoku University, 6-3 Aoba, Aramaki, Aoba-ku, Sendai 980-8578, Japan.}
%\email{ryuichi.sato.d4@tohoku.ac.jp}
%\email{r1sato@ms.u-tokyo.ac.jp}
\email{rsato@fukuoka-u.ac.jp}

\date{\today}

\thanks{\textbf{Acknowledgment.}%\quad
\textrm{
GA is supported by JSPS KAKENHI Grant Number JP20H01812, JP18K18715, JP16H03946 and JP20H00117, JP17H01095 
and by the Alexander von Humboldt Foundation and by the Carl Friedrich von Siemens Foundation. KI and RS are supported in part by KAKENHI Grant Number 19H05599.}}


\begin{abstract}
This paper presents an integrated framework to construct \emph{local-energy solutions} to 
fairly general nonlinear diffusion equations for initial data growing at infinity under suitable assumptions on local-energy estimates for approximate solutions. A delicate issue for constructing local-energy solutions resides in the \emph{identification of weak limits of nonlinear terms} for approximate solutions in a limiting procedure. Indeed, such an identification process often needs the \emph{maximal monotonicity} of nonlinear elliptic operators (involved in the doubly-nonlinear equations) as well as uniform estimates for approximate solutions; however, even the monotonicity is violated due to a \emph{localization} of the equations, which is also necessary to derive local-energy estimates for approximate solutions. In the present paper, such an inconsistency will be systematically overcome by reducing the original equation to a localized one, where a (no longer monotone) localized elliptic operator will be decomposed into the sum of a maximal monotone one and a perturbation, and by integrating all the other relevant processes. Furthermore, the general framework developed in the present paper will also be applied to the \emph{Finsler porous medium and fast diffusion equations}, which are variants of the classical PME and FDE and also classified as a doubly-nonlinear equation.
\end{abstract}

\maketitle

\section{Introduction}

There have been numerous contributions to nonlinear diffusion equations 
such as the porous medium equation (PME for short), fast diffusion equation (FDE for short), $p$-Laplace parabolic equation, Stefan problem, Richards equation and the so-called \emph{doubly-nonlinear parabolic equation}, which is a unified form of the aforementioned equations. Some of these equations exhibit substantially different dynamics in solutions (e.g., loss of regularity) from that of the normal diffusion equation due to the degeneracy and singularity of diffusion machinery. For instance, propagation speeds of interfaces for degenerate diffusion equations (e.g., PME) are often finite, and then, some singularity emerges on such an interface (and hence, weak notions of solutions are essentially needed for those equations). Indeed, one can clearly observe such a singularity in the so-called ZKB (or Barenblatt) solutions, which are self-similar solutions of an explicit form and whose regularity is often at most H\"older continuous on the interface (see, e.g.,~\cite{Vaz}). On the other hand, to prove existence of a (weak) solution, due to the degeneracy and singularity of nonlinear diffusion terms, an additional task arises to \emph{identify weak limits of nonlinear terms}, that is, let $u$ be a \emph{weak} limit of a sequence $(u_n)$ (e.g., of approximate solutions) and let $f$ be a nonlinear term (e.g., $f(u)$ is a power of $u$ for the PME and FDE); the task is then to discuss a crucial step, whether a \emph{weak} limit of $f(u_n)$ coincides with $f(u)$ or not. However, the loss of classical regularity often prevents us to verify (pre)compactness of approximate solutions in a sufficiently strong topology for identifying the weak limit.

The PME and FDE  in a weak form involve only a power of the unknown function $u$ itself as a nonlinearity, thanks to the linearity of the Laplacian along with integration by parts. As for the $p$-Laplace parabolic equation, a gradient nonlinearity $f(\nabla u)$ remains even in a weak form. Moreover, the doubly-nonlinear equation combines nonlinearities of both equations. A typical doubly-nonlinear equation may be a unified form of the PME/FDE and the $p$-Laplace parabolic equation (see \S \ref{Ss:dnp} below for more details). We shall also later treat the \emph{Finsler porous medium and fast diffusion equations}, which are variants of the PME and FDE, respectively, and also classified as a sort of doubly-nonlinear equation (see \S \ref{Ss:Finsler} below). Hence issues on the identification of weak limits always arise to prove existence of solutions for these nonlinear diffusion equations. A main purpose of this paper is to present a general approach to settle such an issue on the identification of weak limits for doubly-nonlinear parabolic equations.

Identification of weak limits of nonlinear terms has been studied in various scenes and the so-called \emph{Minty's trick} is widely used to cope with the issue. Minty's trick is based on a closedness (in a weak topology) of maximal monotone graphs. Roughly speaking, whenever $f$ enjoys some ``monotonicity'' and either $u_n$ or $f(u_n)$ is strongly convergent with some ``duality'' of topologies, one can identify the (weak) limit of $f(u_n)$ with $f(u)$ (see Proposition \ref{P:Minty} below for more details). Therefore the reduction of each PDE to a functional analytic setting plays a crucial role. Indeed, the maximal monotonicity of the even classical Laplacian relies on the boundary condition as well as the choice of base spaces. Another device for the issue relies on the well-known fact that pointwise and weak limits coincide each other in Lebesgue spaces, and it enables us to identify weak limits for general continuous functions $f$ and sequences $(u_n)$ converging pointwisely. \begin{comment} However, both approaches do not work out directly for doubly-nonlinear problems, which involve a gradient nonlinearity $f(\nabla u)$ but whose approximate solutions may be uniformly bounded in a H\"older space at most.}\end{comment} Therefore we may need at least either a fine monotone structure of the equation or higher regularity estimates which yield pointwise convergence of $(u_n)$ and their gradients.

It turns more delicate and critical in which (functional analytic) setting one should work on the task when we need to localize the problem. 
For instance, such situations actually occur to tackle construction of a \emph{local-energy solution} for growing initial data, which may not lie on standard Lebesgue and Sobolev spaces over the whole of domains. Studies on solutions growing at spatial infinity date back to the celebrated work of Tychonoff~\cite{Tyc35}, where it is proved that the Cauchy problem for the heat equation admits a unique solution on $(0,1/(4\Lambda))$ 
for initial data $\mu$ as a (signed) Radon measure satisfying 
$$
\int_{\R^\dim}e^{-\Lambda|x|^2}\d |\mu(x)|<\infty
$$
where $\Lambda>0$. 
Moreover, the result has been extended to the PME in~\cite{BeCrPi84}, the FDE in~\cite{HerPi85} as well as the $p$-Laplace parabolic equation in~\cite{DiHe89} (see also~\cite{DiBook}), and also to a doubly-nonlinear equation unifying these equations in~\cite{Ishige96}. To construct such non-integrable solutions, we usually need some \emph{localization} of equations to establish local-energy estimates; however, such localization may violate the monotonicity of nonlinear elliptic operators involved in the equations, and then, there arises a problem in identifying their weak limits. 
Concerning the degenerate $p$-Laplace equation, in~\cite{DiBook}, $C^{1,\alpha}$-estimates are established and then applied to identity the weak limit of gradient nonlinearity with aid of Ascoli's lemma. 

In the present paper, we shall develop a general framework which enables us to localize problems as well as to identify weak limits of nonlinear terms of approximate solutions for fairly general doubly-nonlinear parabolic equations 
for growing initial data without relying on higher-order estimates (e.g., $C^{1,\alpha}$-estimates) and specific structures of equations. 
Moreover, we shall also apply the theory to the Finsler PME and FDE with growing initial data.

\subsection{General framework for doubly-nonlinear parabolic equations}\label{Ss:DNP}

We shall set up a framework to construct \emph{local-energy solutions} of the Cauchy problem for fairly general doubly-nonlinear parabolic equations with growing initial data,
\begin{alignat}{4}
 \partial_t v &= \mathrm{div} \, a(x,t,\nabla u), \ v \in \beta(u) \ &&\mbox{ in } \Rd \times (0,\Sl),\label{dnp}\\
 v &= \mu &&\mbox{ on } \Rd \times \{0\},\label{dnp-ic}
\end{alignat}
where $\mu$ is a Radon measure in $\Rd$ and $a = a(x,t,\xi) : \Rd \times \R_+ \times \Rd \to \Rd$ is a \emph{Carath\'eodory function}, i.e., measurable in $(x,t)$ and continuous in $\xi$, satisfying
\begin{align}
% |\xi|^p &\leq C(\hat{a}(x,t,\xi) + k(x,t)) \quad && \mbox{ for } \ x \in \Rd,\ t > 0, \ \xi \in \Rd,
& \xi \mapsto a(x,t,\xi) \ \mbox{ is monotone in } \R^\dim \ \mbox{ for a.e.} \ x \in \Rd, \ t > 0,\label{a-mono}\\
%& |\xi|^p \leq C \left( a(x,t,\xi) \cdot \xi + k(x,t)\right) \mbox{ for } \ \xi \in \Rd \ \mbox{ and a.e. } x \in \Rd, \ t>0,\label{a-bdd}\\
& |a(x,t,\xi)|^{p'} \leq C(|\xi|^p+k(x,t)) \ \mbox{ for } \ \xi \in \Rd \ \mbox{ and a.e. } x \in \Rd, \ t>0 \label{a-bdd}
\end{align}
for some $1 < p < +\infty$, $0 \leq C < +\infty$ and $k \in L^1_{\rm loc}(\Rd\times\R_+)$. 
Here $p'$ is the H\"older conjugate of $p$, that is, $1/p+1/p'=1$. 
Moreover, $\beta$ is a maximal monotone graph in $\R \times \R$ (then one can always write $\beta = \partial \hat{\beta}$ for some proper lower semicontinuous convex function $\hat\beta : \R \to (-\infty,+\infty]$) and complies with the assumption,
\begin{align}
\hat\beta \ \mbox{ is strictly convex.}\label{beta-conv}
\end{align}
We can assume, without any loss of generality, that
$$
0 \in D(\beta), \quad \beta(0) \ni 0  \quad \mbox{and} \quad 
\hat\beta, \, (\hat\beta)^* \geq 0, %, \quad a(\cdot,\cdot,0) = 0 \ \mbox{ and } \ \hat a \geq 0,
$$
where $(\hat{\beta})^*$ stands for the Legendre-Fenchel transform (or convex conjugate) of $\hat \beta$, by translation. Throughout this paper, we denote by $\BR$ the open ball in $\Rd$ centered at the origin with radius $R>0$ in terms of an \emph{arbitrary} norm of $\Rd$.
In what follows, we are concerned with \emph{local-energy solutions} of \eqref{dnp}, \eqref{dnp-ic} defined by

\begin{definition}[Local-energy solution of \eqref{dnp}, \eqref{dnp-ic}]\label{D:sol-dnp}
For $\Sl > 0$, a pair of measurable functions $(u,v) : \Rd \times (0,\Sl) \to \R^2$ {\rm (}or just $u : \Rd \times (0,\Sl) \to \R${\rm )} is called a \emph{local-energy solution} of \eqref{dnp}, \eqref{dnp-ic} on $(0,\Sl)$ if the following conditions are all satisfied\/{\rm :}
\begin{enumerate}
 \item[\rm (i)] It holds that
\begin{align*}
 u &\in L^p(\vep,T;W^{1,p}(\BR)),\\
 v &\in L^1(\BR \times (0,T)) \cap C_{\rm weak}([\vep,T];L^{p'}(\BR)),\\
 a(x,t,\nabla u) &\in L^1(\BR \times (0,T)) \cap L^{p'}(\BR \times (\vep,T)),
\end{align*}
for any $R > 0$ and $0 < \vep < T < \Sl$. Here $C_{\rm weak}$ stands for the class of weakly continuous functions {\rm (}see Notation below{\rm )}.
 \item[\rm (ii)] It also holds that
\begin{align}
- \int^t_0 \int_{\Rd} v \partial_t \psi \, \d x \d \tau
 + \int_{\Rd} v(\cdot,t) \psi(\cdot,t) \, \d x
  - \int_{\Rd} \psi(\cdot,0) \, \d \mu(x)\nonumber\\
+ \int^t_0 \int_{\Rd} a(x,\tau,\nabla u) \cdot \nabla \psi \, \d x \d \tau = 0\label{eq}
\end{align}
 for all $\psi \in C^\infty_c([0,\Sl)\times \R^\dim)$ and $0 < t < \Sl$, and moreover,
$$
v(x,t) \in \beta(u(x,t)) \quad \mbox{ for a.e. } (x,t) \in \Rd \times (0,\Sl).
$$
\end{enumerate} 
\end{definition}

\begin{remark}
{\rm
{\rm (i)} 
By subtraction, we can also derive from a weak form of \eqref{dnp}, that is, 
\begin{align*}
- \int^{t_2}_{t_1} \int_{\Rd} v \partial_t \psi \, \d x \d t
 + \int_{\Rd} v(\cdot,t_2) \psi(\cdot,t_2) \, \d x
  - \int_{\Rd} v(\cdot,t_1) \psi(\cdot,t_1) \, \d x\nonumber\\
+ \int^{t_2}_{t_1} \int_{\Rd} a(x,t,\nabla u) \cdot \nabla \psi \, \d  x \d t = 0
\end{align*}
for all $\psi \in C^\infty_c((0,\Sl)\times \R^\dim)$ and $0 < t_1 < t_2 < \Sl$.
Moreover, setting $\psi(\cdot,t) \equiv \varphi \in C_c(\Rd)$ for $t$ close to zero, 
one finds from \eqref{eq} that $v(\cdot,t)\rightarrow \mu$ weakly star in $\mathcal M(\R^\dim)$ as $t\to 0_+$, 
which corresponds to the initial condition \eqref{dnp-ic}.
\vspace{3pt}
\newline
{\rm (ii)} 
Let $(u,v)$ be a local-energy solution of \eqref{dnp}, \eqref{dnp-ic} on $(0,\Sl)$. By definition, $v$ belongs to $W^{1,{p'}}_{\rm loc}(0,\Sl;W^{-1,p'}(\BR))$, where $W^{-1,p'}(\BR)$ is the dual space of $W^{1,p}_0(\BR)$, for any $R > 0$. Indeed, for $t \in (0,T)$, define $\xi(t) \in W^{-1,p'}(\BR)$ by
$$
\left\langle \xi(t), w \right\rangle_{W^{1,p}_0(\BR)}
= \int_{\BR} a(x,t,\nabla u(x,t)) \cdot \nabla w(x) \, \d x
$$
for $w \in W^{1,p}_0(\BR)$. Then it follows that $\xi \in L^{{p'}}_{\rm loc}(0,\Sl;W^{-1,p'}(\BR))$. Substituting $\varphi(x,t) = w(x) \rho(t)$ to \eqref{eq} (by density) for any $\rho \in C^\infty_c(0,\Sl)$ and using the arbitrariness of $w \in W^{1,p}_0(\BR)$, we find that
$$
\int^{\Sl}_0 v(t) \partial_t \rho(t) \, \d t = \int^{\Sl}_0 \xi(t) \rho(t) \, \d t \ \mbox{ in } W^{-1,p'}(\BR)
$$
for a.e.~$t \in (0,\Sl)$. Hence we obtain $v \in W^{1,{p'}}_{\rm loc}(0,\Sl;W^{-1,p'}(\BR))$ for any $R > 0$.
}
\end{remark}

To construct a local-energy solution of \eqref{dnp}, \eqref{dnp-ic}, we start with approximation. 
Let $(\mu_n)$ be a sequence in $C^\infty_c(\Rd)$ such that $\mathrm{supp}\,\mu_n \subset B_n$ and 
\begin{equation*}
\mu_n \to \mu \quad \mbox{ weakly star in } \mathcal M(\R^\dim),
\end{equation*}
that is, 
\begin{equation}\label{ini-hypo}
\int_{\Rd} \varphi \mu_n \, \d x \to \int_{\Rd} \varphi \,\d \mu(x) \quad \mbox{ for } \ \varphi \in C_c(\Rd)
\end{equation}
(see Lemma \ref{L:meas} in Appendix). Then we shall consider the approximate problem,
\begin{alignat}{4}
 \partial_t v_n &= \mathrm{div} \, a(x,t,\nabla u_n), \ v_n \in \beta(u_n) \ &&\mbox{ in } B_n \times (0,\Sl),\label{dnp-n}\\
 u_n &= 0 &&\mbox{ on } \partial B_n \times (0,\Sl),\label{dnp-bc-n}\\
 v_n &= \mu_n &&\mbox{ on } B_n \times \{0\}. \label{dnp-ic-n}
\end{alignat}
We shall work along with the following basic assumptions:
\begin{itemize}
  \item[({\bf A0})]
  (Existence of approximate energy solutions)
  For each $n \in \N$ {\rm (}large enough{\rm )}, there exists an energy solution $(u_n,v_n)$ to \eqref{dnp-n}--\eqref{dnp-ic-n} on $[0,\Sl]$, that is,
  \begin{align}
  u_n \in L^p(0,\Sl;W^{1,p}_0(B_n)),\nonumber\\
  v_n \in W^{1,p'}(0,\Sl;W^{-1,p'}(B_n)) \cap C_{\rm weak}([0,\Sl];L^{p'}(B_n)),\nonumber\\
  v_n(x,t) \in \beta(u_n(x,t)) \ \mbox{ for a.e. } (x,t) \in \Rd \times (0,\Sl),\nonumber\\
  \left\langle \partial_t v_n(t), w \right\rangle_{W^{1,p}_0(B_n)}
  + \int_{B_n} a(x,t,\nabla u_n(x,t)) \cdot \nabla w(x) \, \d x = 0\label{dnp:wf}
  \end{align}
  for any $w \in W^{1,p}_0(B_n)$ and a.e.~$t \in (0,\Sl)$. 
  \item[({\bf A1})] 
  (Local-energy estimates)
  For any $R > 0$ and $0 < t_1 < t_2 < S$, there exist constants $\delta>p$ and $M > 0$ such that 
  \begin{equation}\label{un-bdd}
  %\sup_{t \in (t_1,t_2)} \left( \int_{\BR}|u_n(\cdot,t)| \, \d x \right) +
  \int^{t_2}_{t_1} \left( \int_{\BR}|u_n| \, \d x \right)^\delta \, \d t + 
  \int^{t_2}_{t_1} \int_{\BR} \left( |v_n|^{p'} + |\nabla u_n|^p \right) \, \d x \d t \leq M
  \end{equation}
  for any $n \in \N$ greater than $R$.
  \item[({\bf A2})] 
  (Uniform integrability around $t = 0$)
  For each $R > 0$, it holds that
  $$
  \sup_{n \in \N} \left( \int^t_0 \int_{\BR} \left( |v_n| + |a(x,\tau,\nabla u_n)| \right) \, \d x \d \tau \right) \to 0 \quad \mbox{ as } \ t \to 0_+.
  $$
  \end{itemize}
%
\begin{comment}
\begin{remark}[Existence of approximate energy solutions]
{\rm
Assumption (A0) can be checked for well-prepared approximate data $\mu_n$ (see, e.g.,~\cite{A06}). However, we here assume it to avoid discussing existence of such well-prepared data in particular for non-smooth $\beta$ with $D(\beta) \neq \R$. On the other hand, one can always check (A0) for smooth $\beta$ such as power and exponential nonlinearities under a smoothness of $t \mapsto a(x,t,\xi)$.
}
\end{remark}
\end{comment}
%
Then our result reads, 
\begin{theorem}[Construction of local-energy solution]\label{T:loc}
In addition to \eqref{a-mono}--\eqref{beta-conv}, assume that {\rm (A0)}--{\rm (A2)} hold for some $\Sl \in (0,+\infty]$. For each $n \in \N$, let $u_n$ be the energy solution of \eqref{dnp-n}--\eqref{dnp-ic-n} on $[0,\Sl]$. Then there exist a {\rm (}not relabeled{\rm )} subsequence of $(n)$ and a local-energy solution $(u,v) : \Rd \times (0,\Sl) \to \R^2$ of \eqref{dnp}, \eqref{dnp-ic} in $(0,\Sl)$ such that
\begin{alignat}{4}
 u_n &\to u \quad &&\mbox{ weakly in } L^p(t_1,t_2;W^{1,p}(\BR)),\label{tloc:1}\\
 & &&\mbox{ a.e.~in } \Rd \times (0,\Sl),\label{tloc:3}\\
 v_n &\to v \quad &&\mbox{ weakly star in } L^\infty(t_1,t_2;L^{p'}(\BR)),\label{tloc:2}\\
 a(\cdot,\cdot,\nabla u_n) &\to a(\cdot,\cdot,\nabla u) \quad &&\mbox{ weakly in } L^{p'}(\BR \times (t_1,t_2))^N\label{tloc:4}
\end{alignat}
for any $R>0$ and $0 < t_1 < t_2 < \Sl$. 
\end{theorem}
Theorem~\ref{T:loc} provides a general framework for constructing local-energy solutions to doubly-nonlinear parabolic equations, and then, it enables us to concentrate on constructing approximate solutions on bounded domains and establishing local-energy estimates (without higher regularity ones) for the approximate solutions. These processes will depend more deeply on the structure of each equation and may reveal more quantitative information such as growth and local existence time of solutions. Such remaining processes will also be discussed below for some specific PDEs.%}
\subsection{Finsler PME and FDE}\label{Ss:Finsler}

We shall apply the preceding general theory to the \emph{Finsler porous medium equation} and \emph{Finsler fast diffusion equation}, which are variants of the standard porous medium equation ($1 < m < +\infty$ in \eqref{pmfd-1} and $1 < q < 2$ in \eqref{pmfd-2} below) and fast diffusion equation ($0 < m < 1 $ in \eqref{pmfd-1} and $2 < q < +\infty$ in \eqref{pmfd-2} below), respectively, of the form
\begin{equation}
\partial_t w = \Delta \left( |w|^{m-1}w \right) \ \mbox{ in } \Rd \times (0,+\infty) \label{pmfd-1}
\end{equation}
or an equivalent form
\begin{equation}\label{pmfd-2}
\partial_t \left( |u|^{q-2}u \right) = \Delta u \ \mbox{ in } \Rd \times (0,+\infty)
\end{equation}
with the choice $u := |w|^{m-1}w$ along with $q := (1+m)/m$. To be more precise, we shall consider the Cauchy problem,
\begin{alignat}{4}
\partial_t \left(|u|^{q-2}u\right) &= \Lap u \quad &&\mbox{ in } \R^\dim \times (0,+\infty),\label{pde}\\
|u|^{q-2}u&=\mu \quad &&\mbox{ in } \R^\dim \times \{0\},\label{ic}
\end{alignat}
where $1 < q < +\infty$. Here $\Lap$ is the so-called \emph{Finsler Laplacian} given by
\begin{align*}
\Lap u &:= \mathrm{div}\, \left( H(\nabla u)\nabla_\xi H(\nabla u) \right)\\
&= \sum_{j=1}^\dim \frac{\partial}{\partial x_j} \left( H(\nabla u) \frac{\partial H}{\partial \xi_j}(\nabla u) \right),
\end{align*}
where $H\in C(\Rd)\cap C^1(\Rd\setminus\{0\})$ is a (possibly non-Euclidean) norm of $\Rd$, that is 
\begin{equation}
\label{eq:1.1}
\left.
\begin{array}{l}
 H\ge 0 \mbox{ in } \Rd \ \mbox{ and } \ H(\xi)=0 \ \mbox{ if and only if } \ \xi=0,\\
 H(\alpha\xi)=|\alpha|H(\xi) \ \mbox{ for } \ \xi\in\Rd \ \mbox{ and } \ \alpha\in\Rd,\\
 H(\xi_1 + \xi_2) \leq H(\xi_1) + H(\xi_2) \ \mbox{ for } \ \xi_1,\xi_2 \in \Rd.
\end{array}
\right\}
\end{equation}
We denote by $H_0$ the dual norm  of $H$ defined by
$$
H_0(x):=\sup_{\xi\in\Rd\setminus\{0\}}\frac{x\cdot\xi}{H(\xi)}. 
$$
Then it follows immediately that
\begin{equation}
\label{eq:1.2}
|x\cdot\xi|\le H_0(x)H(\xi),
\quad
H(\xi)=\sup_{x\in\Rd\setminus\{0\}}\frac{x\cdot\xi}{H_0(x)}. 
\end{equation}
The Finsler Laplacian can be regarded as an anisotropic variant of the Laplacian, and hence, it may not comply with a \emph{strong monotonicity}, which is a typical feature of the $p$-Laplacian, e.g.,
$$
\left\langle -\Delta_p u + \Delta_p v, u-v \right\rangle_{W^{1,p}(\R^\dim)} \geq \omega_p \|\nabla u - \nabla v\|_{L^p(\R^\dim)}^p \quad \mbox{ for } \ u,v \in W^{1,p}(\R^\dim)
$$
for some constant $\omega_p > 0$, provided that $p \geq 2$. Actually, one can construct a counterexample by choosing $H(\cdot)$ as a non-Euclidean norm, e.g., $H(\xi) = (\sum_{j=1}^\dim |\xi_j|^q)^{1/q}$ for $\xi = (\xi_1,\ldots,\xi_\dim) \in \R^\dim$ with any $q \in (1,2)$.

Throughout this paper, we assume that 
\begin{equation}
\label{eq:1.3}
\mbox{the open unit ball } \{x \in \Rd \colon H(x) < 1\} \mbox{ is strictly convex},
\end{equation}
which is equivalent to $H_0\in C^1({\bf R}^N\setminus\{0\})$ (see \cite[Corollary~1.7.3]{S}). 

In general, the Finsler Laplacian is a nonlinear (but uniform) elliptic operator (cf.~$\Lap$ coincides with the usual Laplacian, if $H(\cdot)$ is the Euclidean norm), and moreover, \eqref{pde} is also classified as a sort of doubly-nonlinear problem due to the \emph{gradient nonlinearity} intrinsic to the Finsler Laplacian. Moreover, we can reduce the Finsler PME/FDE \eqref{pde} to \eqref{dnp} by setting $a(x,t,\xi) = H(\xi) \nabla_\xi H(\xi)$ and $\beta(u) = |u|^{q-2}u$, whose potential function is given by
$$
%\hat a(x,t,\xi) = \frac 1 2 H(\xi)^2 \quad \mbox{ and } \quad 
\hat \beta(u) = \frac 1 q |u|^q.
$$
On the other hand, the dynamics in solutions for \eqref{pde} shares much with the PME/FDE. For instance, the authors~\cite{AIS} exhibit a \emph{Finsler ZKB {\rm(}Zel'dovich-Kompaneets-Barenblatt{\rm)} solution} to \eqref{pde} for $1<q<2$ with a parameter $C > 0$,
$$
{\mathcal U}_{H_0}(x,t;C):= U(H_0(x),t;C) = t^{-\frac \alpha {q-1}}\left(C-kH_0(x)^2 t^{-2\beta}\right)_+^{\frac 1 {2-q}},
$$
where $U(r,t;C)$ is a radial profile of the ZKB solution $\mathcal{U}(x,t;C)$ (i.e., $\mathcal U(x,t;C) = U(|x|,t;C)$) for the PME \eqref{pmfd-2}, $(s)_+:=\max\{s,0\}$ and
$$
\alpha:=\frac{N}{2+N\frac{2-q}{q-1}},\quad
\beta:=\frac{\alpha}{N},\quad
k:=\frac{\alpha(2-q)}{2N},
$$
and then ${\mathcal U}_{H_0}(x,t;C)$ turns out to be singular on the interface 
$\{x \in \Rd \colon H_0(x) = C^{1/2}k^{-1/2}t^{\beta}\}$, provided that $1 < q < 2$.

In what follows, we write 
$$
\B_R := \{x \in \R^\dim \colon H_0(x) < R\} \quad \mbox{ for } \ R > 0.
$$
The notion of local-energy solution for the Cauchy problem \eqref{pde}, \eqref{ic} can be analogously defined as follows:
\begin{definition}[Local-energy solution of \eqref{pde}, \eqref{ic}]\label{D:sol}
For $\Sl \in (0,+\infty]$, a measurable function $u : \Rd \times (0,\Sl) \to \R$ is called a \emph{local-energy solution} of \eqref{pde}, \eqref{ic} on $(0,\Sl)$ if the following conditions are all satisfied\/{\rm :}
\begin{enumerate}
 \item[\rm (i)] It holds that
\begin{align*}
 u &\in L^2(\vep,T;H^1(\B_R)),\\
 u^{q-1} &\in L^1(\FBR \times (0,T)) \cap C_{\rm weak}([\vep,T];L^2(\FBR)),\\
 H(\nabla u) &\in L^1(\FBR \times (0,T)) \cap L^2(\FBR \times (\vep,T))
\end{align*}
for any $R > 0$ and $0 < \vep < T < \Sl$.
% \item[\rm (ii)] For all $0 < t_1 < t_2 < \Sl$, it holds that
%\begin{align}
%- \int^{t_2}_{t_1} \int_{\Rd} u^{q-1} \partial_t \psi \, \d x \d t
% + \int_{\Rd} u(x,t_2)^{q-1} \psi(x,t_2) \, \d x
%  - \int_{\Rd} u(x,t_1)^{q-1} \psi(x,t_1) \, \d x\nonumber\\
%+ \int^{t_2}_{t_1} \int_{\Rd} H(\nabla u) \dH(\nabla u) \cdot \nabla \psi \, \d% x \d t = 0\label{eq:4.12}
%\end{align}
% for all $\psi \in C^\infty_c((0,\Sl)\times \R^d)$.
 \item[\rm (ii)] For all $0 < t < \Sl$, it also holds that
\begin{align*}
- \int^t_0 \int_{\Rd} u^{q-1} \partial_t \psi \, \d x \d \tau
 + \int_{\Rd} u(\cdot,t)^{q-1} \psi(\cdot,t) \, \d x
  - \int_{\Rd} \psi(\cdot,0) \, \d \mu(x)\nonumber\\
+ \int^t_0 \int_{\Rd} H(\nabla u) \dH(\nabla u) \cdot \nabla \psi \, \d x \d \tau = 0
\end{align*}
 for all $\psi \in C^\infty_c(\R^\dim\times[0,\Sl))$.
\end{enumerate} 
\end{definition}

In what follows, we shall state results on existence of local-energy solutions for the Finsler porous medium and fast diffusion equations for growing initial data; briefly stated, in the porous medium case, the growth of initial data at infinity will be restricted in a quantitative way to assure existence of local energy solutions, and moreover, at some critical growth, solutions exist locally in time, but may blow up in finite time (see Theorem \ref{T:pme} and Remark \ref{R:opt} below). On the other hand, in the fast diffusion case, local-energy solutions always exist globally in time for any (non-negative) initial data with no growth restriction (see Theorem \ref{T:fde} below). Such a clear contrast may be intuitively interpreted in terms of the singularity and degeneracy of diffusion coefficients for both cases as $|u| \to +\infty$. Here and henceforth, we set
\begin{equation}\label{kappa_r}
\kappa_p := 2p - N \frac{q-2}{q-1} \quad \mbox{ for } \ p \geq 1.
\end{equation}
Applying Theorem \ref{T:loc}, we have

\begin{theorem}[Finsler fast diffusion]\label{T:fde}
Let $q > 2$ satisfy $\kappa_1 > 0$, that is,
\begin{equation}\label{hypo-q}
q < \dfrac{2(N-1)}{(N-2)_+} 
\end{equation}
and let $\mu \in \mathcal M(\R^\dim)$ be any non-negative Radon measure. Then \eqref{pde}, \eqref{ic} admits a local-energy solution $u : \Rd \times (0,+\infty) \to [0,+\infty)$ on $(0,+\infty)$ defined in Definition \ref{D:sol} such that
$$
u \in L^\infty_{\rm loc}(\Rd \times (0,+\infty)).
$$
Moreover, there exists a constant $C$ depending only on $N,q$ such that
\begin{align}
\sup_{0<\tau<t} \|u(\cdot,\tau)^{q-1}\|_{L^1(\B_R)} &\leq C \mu(\B_{2R}) + C \left( \frac t {R^{\kappa_1}}\right)^{\frac{q-1}{q-2}},\label{fde:2}\\
\|u(\cdot,t)\|_{L^\infty(\B_R)} &\leq C t^{-\frac N {\kappa_1(q-1)}} \left(\sup_{\frac t8 < \tau < t} \|u(\cdot,\tau)^{q-1}\|_{L^1(\B_{2R})}\right)^{\frac 2 {\kappa_1(q-1)}}\nonumber\\
&\quad + C  \left( \frac t {R^2}\right)^{\frac{1}{q-2}},\label{fde:1}\\
\int^t_0 \|H(\nabla u) \|_{L^1(\B_R)} \, \d \tau &\leq C t^{\frac 1 2} R^{\frac{N(q-2)}{2(q-1)}} \left( \mu(\B_{2R}) + t^{\frac{q-1}{q-2}} R^{N-2\frac{q-1}{q-2}} %\sup_{0<\tau<t} \|u(\cdot,\tau)^{q-1}\|_{L^1(\B_{2R})} 
\right)^{\frac q{2(q-1)}} \nonumber\\
&\quad + C t^{\frac{q-1}{q-2}} R^{N-\frac q{q-2}}
\label{fde:3}
\end{align} 
for any $R > 0$ and $t > 0$. 

For any $q > 2$ {\rm(}without any further restriction{\rm )}, if $\d \mu(x) = v_0(x)\, \d x$ for some non-negative $v_0 \in L^p_{\rm loc}(\R^\dim)$ with $p \geq 1$ satisfying $\kappa_p > 0$, that is,
\begin{equation}\label{hypo-p}
p > \frac{N(q-2)}{2(q-1)},
\end{equation}
then the same assertions except \eqref{fde:1} also hold true, and moreover, it holds that
\begin{align*}
\sup_{0<\tau<t}\|u(\cdot,\tau)^{q-1}\|_{L^p(\B_R)}^{p} &\leq C \|v_0\|_{L^p(\B_{2R})}^p + C \left(\frac{t^p}{R^{\kappa_p}}\right)^{\frac{q-1}{q-2}},\\
\|u(\cdot,t)\|_{L^\infty(\B_R)} &\leq C t^{-\frac N {\kappa_p(q-1)}} \left(\sup_{\frac t8 < \tau < t} \|u(\cdot,\tau)^{q-1}\|_{L^p(\B_{2R})}\right)^{\frac{2p}{\kappa_p(q-1)}} + C \left(\frac t {R^2}\right)^{\frac 1{q-2}}
\end{align*}
for any $R > 0$ and $t > 0$.
\end{theorem}

The assumption \eqref{hypo-q} (equivalently, $m > m_c = (\dim-2)_+/\dim$ in the form \eqref{pmfd-1}) is essentially needed to guarantee local boundedness of weak solutions for the case $\kappa_1 = 0$. We refer the reader to~\cite[Remark 3.6]{DK92}, where a counter example is provided for the classical FDE, and moreover, it can be extended to the Finsler FDE with the aid of~\cite[Theorem 1.1]{AIS}.

As for the Finsler porous medium equation, one has

\begin{theorem}[Finsler porous medium]\label{T:pme}
Let $1 < q < 2$ and set $d = \frac{2-q}{q-1} > 0$. Let $\mu \in \mathcal M(\R^\dim)$ be a {\rm (}signed{\rm)} Radon measure such that
$$
\normtre{\mu}_r := \sup_{R \geq r} \left( R^{-\frac {\kappa_1} d} |\mu|(\B_R) \right) < +\infty
$$
for some {\rm (}and any{\rm )} $r > 0$. Then there exists $\T(\mu) \in (0,+\infty]$ such that \eqref{pde}, \eqref{ic} admits a local-energy solution $u : \Rd \times (0,\T(\mu)) \to \R$ on $(0,\T(\mu))$ in the sense of Definition \ref{D:sol} such that
$$
u \in L^\infty_{\rm loc}(\Rd \times (0,\T(\mu)).
$$
Here $\T(\mu)$ is given by
\begin{equation}\label{Tmu}
\T(\mu) = \begin{cases}
	 c \left( a_\mu \right)^{-d} &\mbox{ if } \ a_\mu := \displaystyle\lim_{r \to +\infty} \normtre{\mu}_r > 0,\\
	+\infty &\mbox{ if } \ a_\mu = 0
       \end{cases}
\end{equation}
for some constant $c > 0$. Moreover, for each $r > 0$, set 
$$
\Tr(\mu) = c \normtre{\mu}_r^{-d}.
$$
Then there exists a constant $C > 0$ such that
\begin{align}
\normtre{u(\cdot,t)^{q-1}}_r &\leq C \normtre{\mu}_r,\label{pme:1}\\
\|u(\cdot,t)\|_{L^\infty(\B_R)} &\leq C t^{-\frac N {\kappa_1(q-1)}} R^{\frac 2 {d(q-1)}} \normtre{\mu}_r^{\frac 2 {\kappa_1(q-1)}},\label{pme:2}\\
\int^t_0 \|H(\nabla u)\|_{L^1(\B_R)} \, \d \tau &\leq C t^{\frac 1 {\kappa_1}} R^{1+\frac {\kappa_1} d} \normtre{\mu}_r^{1+\frac d {\kappa_1}}\label{pme:3}
\end{align}
for any $R > r$ and $0 < t < \Tr(\mu)$. Here $\normtre{\,\cdot\,}_r$ is given by
$$
\normtre{f}_r := \sup_{R \geq r} \left( R^{-\frac {\kappa_1} d} \int_{\B_R} |f(x)| \, \d x \right) \quad \mbox{ for } \ f \in L^1_{\rm loc}(\Rd).
$$
\end{theorem}

\begin{remark}[Optimality of the growth condition]\label{R:opt}
{\rm
B\'enilan et al~\cite{BeCrPi84} exhibits a finite-time blow-up of some explicit solution for the classical PME. Thanks to~\cite[Theorem 1.1]{AIS}, by replacing the Euclidean norm $|\cdot|$ with $H_0(\cdot)$, it can be transformed to an explicit solution of the Finsler PME with an initial datum $\mu$ growing at the critical rate (to be more precise, $a_\mu = \lim_{r\to+\infty}\normtre{\mu}_r$ is finite and positive) and local-existence time proportional to $(a_\mu)^{-d}$.
}
\end{remark}

To prove these theorems, we shall often use the following useful properties of the norm $H(\cdot)$ as well as the dual norm $H_0(\cdot)$:
\begin{equation}
\label{eq:2.1}
\left.
\begin{array}{ll}
\xi\cdot\nabla_\xi H(\xi)=H(\xi) &\mbox{ for } \ \xi\in\Rd,\\
H_0(\nabla_\xi H(\xi))=1 &\mbox{ for } \ \xi\in\Rd\setminus\{0\}.
\end{array}
\right\}
\end{equation}
Here $\xi\cdot\nabla_\xi H(\xi)$ is supposed to be zero at $\xi=0$. We refer the reader to \cite{CS,FK} for further properties of $H(\cdot)$ and $H_0(\cdot)$. 

%REMARKS ON OTHER FINSLER PDES %(see e.g., \cite{BCS}, \cite{BP}, \cite{BGS}, \cite{CS}, \cite{CFV}, \cite{DPB}, \cite{FK}, \cite{OS1}, \cite{OS2}, \cite{T}, \cite{X} and references therein). 

%\subsection{Structure of the paper} 
\bigskip\noindent
{\bf Structure of the paper.} This paper consists of five sections. In Section \ref{S:ident}, we shall prove Theorem \ref{T:loc}, which is a general part of the theory for constructing local-energy solutions to the Cauchy problem for fairly general doubly-nonlinear parabolic equations of the form \eqref{dnp} with growing initial data. Section \ref{S:FFDE} is devoted to discussing the Finsler fast diffusion equation (i.e., the case where $2 < q < +\infty$) based on the preceding general theory. Main ingredients of this section are local-energy estimates for approximate solutions. Theorem \ref{T:fde} will be finally proved by applying Theorem \ref{T:loc}. Moreover, the Finsler porous medium equation (i.e., $1 < q < 2$) will be treated in Section \ref{S:FPME}, where local-energy estimates are also established and Theorem \ref{T:pme} will be proved. The final section is concerned with other applications and possible extensions of Theorem \ref{T:loc}.

\bigskip\noindent
{\bf Notation.} 
Let $\R_+$ and $\Omega$ 
stand for the open interval $(0,+\infty)$ and a domain in $\R^\dim$, respectively. 
We often write $u(t)$ instead of $u(\cdot,t)$ for each $t\in I$ and $u:\Omega\times\R_+\to\R$. Moreover, $C_{\rm weak}([a,b];X)$ (respectively, $C_{\rm weak\,*}([a,b];X)$) denotes the set of all weakly (respectively, weakly star) continuous functions on $[a,b]$ with values in a Banach space $X$. Furthermore, we denote by $C$ a non-negative constant which is independent of the elements of the corresponding space or set and may vary from line to line. We also often write $f \lesssim g$ for functions or sequences $f,g$, if they comply with the relation
$$
f \leq C g \quad \mbox{uniformly}
$$
for some constant $C > 0$.

\section{General framework to construct local-energy solutions}\label{S:ident}

In this section, we shall prove Theorem \ref{T:loc}. More precisely, under the assumptions (A1) and (A2), which provide \emph{local} boundedness of energy solutions $(u_n,v_n)$ to \eqref{dnp-n}--\eqref{dnp-ic-n} on $[0,\Sl]$, we shall discuss convergence of $(u_n,v_n)$ (up to a subsequence) as $n \to +\infty$ and prove that the limit $(u,v)$ is a local-energy solution of the Cauchy problem. The main difficulty resides in the identification of weak limits of nonlinear terms including the gradient of $u_n$. To overcome it, we shall develop a framework of Minty's trick along with full localization. 

\subsection{Outline}\label{Ss:OL}

Before going to the details, let us give an outline of the proof. Let $(\mu_n)$ be a sequence in $C^\infty_c(B_n)$ such that
$$
\mu_n \to \mu \quad \mbox{ weakly star in } \mathcal M(\R^\dim),
$$
that is,
$$
\int_{\Rd} \varphi \mu_n \, \d x \to \int_{\Rd} \varphi \,\d \mu(x) \quad \mbox{ for } \ \varphi \in C_c(\Rd).
$$

Let $u_n$ be a (smooth for simplicity) solution of the approximate problem \eqref{dnp-n}--\eqref{dnp-ic-n} such that $(u_n)$ is bounded in a local-energy space $L^p(t_1,t_2;W^{1,p}(\BR))$ for any $0 < t_1 < t_2 < \Sl$ and $R > 0$ (see (A1) for more details). 
To localize the equation onto the ball $\BR$, we multiply both sides of \eqref{dnp-n} for $n > R$ by a smooth cut-off (in space) function $\rho$ (supported over $\overline{\BR}$). Then we see that 
$$
\rho \partial_t \beta(u_n) = \rho \, \mathrm{div}\,a(x,t,\nabla u_n) \ \mbox{ in } \BR \times (0,\Sl).
$$
%We shall often regard $u_n$ as a function defined only on $\BR \times (0,\Sl)$ below and n
Furthermore, we multiply 
both sides by a test function $\varphi \in W^{1,p}(\BR)$. 
Note that both $u_n$ and $\varphi$ may not vanish on the boundary $\partial B_R$. 
Since $\rho$ vanishes on $\partial B_R$, we observe that
\begin{align*}
 \int_{\BR} \rho \partial_t \beta(u_n) \varphi \, \d x
 &= \int_{\BR} \rho \,\mathrm{div}\,a(x,t,\nabla u_n) \varphi  \, \d x\\
 &= - \int_{\BR} a(x,t,\nabla u_n) \cdot \nabla (\varphi \rho) \, \d x\\
 &= - \int_{\BR} a(x,t,\nabla u_n) \cdot (\nabla \varphi) \rho \, \d x - \int_{\BR} a(x,t,\nabla u_n) \cdot (\nabla \rho) \varphi \, \d x\\
 &=: - \langle A^t(u_n) , \varphi \rangle_V + \langle F^t(u_n), \varphi \rangle_V,
\end{align*}
where $A^t$ and $F^t$ are operators from $V := W^{1,p}(\BR)$ into $V^*$ defined by
\begin{align*}
\langle A^t(v) , w \rangle_V &= \int_{\BR} a(x,t,\nabla v) \cdot (\nabla w) \rho \, \d x,\\
\langle F^t(v) , w \rangle_V &= - \int_{\BR} a(x,t,\nabla v) \cdot (\nabla \rho) w \, \d x,
\end{align*}
for $v,w \in V$ and $t \in (0,\Sl)$. Hence, $u_n$ solves the following auxiliary evolution equation,
\begin{equation}\label{aux-eq}
\rho \partial_t \beta(u_n) + A^t(u_n) = F^t(u_n) \ \mbox{ in } V^*, \quad 0 < t < \Sl.
\end{equation}
Then  $A^t$ will turn out to be maximal monotone in $V \times V^*$. Therefore, under the local boundedness of $(u_n)$,  the maximal monotonicity enables us to identify the weak  limit of $A^t(u_n)$ by employing standard Minty's trick (see Proposition \ref{P:Minty} below). Furthermore, the limit of $F^t(u_n)$ will also be identified as
$$
\lim_{n \to \infty} \int^{t_2}_{t_1} \langle F^t(u_n), w \rangle_V = - \int^{t_2}_{t_1} \int_{\BR} a(x,t,\nabla u) \cdot (\nabla \rho) w \, \d x \, \d t \quad \mbox{ for } \ w \in V,
$$
whence follows that $u$ solves
$$
\rho \partial_t \beta(u) + A^t(u) = - a(x,t,\nabla u) \cdot (\nabla \rho) \ \mbox{ in } V^*, \quad t_1 < t < t_2.
$$
Setting $\rho \equiv 1$ in $\Br$, we shall obtain 
$$
\partial_t \beta(u) - \mathrm{div}\, a(x,t,\nabla u) = 0 \ \mbox{ in } \Br \times (t_1,t_2)
 $$
 in a weak sense, i.e., $u$ solves the original equation in $B_{R/2}\times(t_1,t_2)$.

 \begin{proposition}[Minty's trick]\label{P:Minty}
  Let $A$ be a {\rm (}possibly multi-valued{\rm )} maximal monotone operator from a Banach space $E$ into its dual space $E^*$. Let $u_n \in D(A)$ and $\xi_n \in A(u_n)$ be such that $u_n \to u$ weakly in $E$, $\xi_n \to \xi$ weakly in $E^*$ and
$$
\limsup_{n \to \infty} \langle \xi_n, u_n \rangle_E \leq \langle \xi , u \rangle_E
$$
for some $u \in E$ and $\xi \in E^*$. Then $u \in D(A)$ and $\xi \in A(u)$. 
Moreover, 
$$
\lim_{n \to \infty} \langle \xi_n, u_n \rangle_E = \langle \xi , u \rangle_E.
$$
 \end{proposition}

\subsection{Localization and functional setting}
% \begin{proof}
%{\bf Step 1. Setting of auxiliary problems.} 

Let us now proceed to the details. Let $R > 0$ be arbitrarily fixed and let $\rho(x) : \Rd \to [0,1]$ be a smooth (say $C^1$) non-negative function satisfying
\begin{equation}\label{rho-hypo}
\mathrm{supp}\, \rho = \overline{\BR}, \quad
\rho \equiv 1 \ \mbox{ on } \Br.%,\quad 
%\dfrac{|\nabla \rho|^p}\rho \in L^\infty(\BR), \quad
%\dfrac{|\nabla \rho|^p}\rho \equiv 0 \ \mbox{ in } \Rd \setminus \BR.
\end{equation}
Define an operator $A^t : \Wpr \to (\Wpr)^*$ for each $t \in (0,\Sl)$ by
\begin{equation}
\label{eq:Q1}
\left\langle A^t (w), e \right\rangle_{ \Wpr}
 = \int_{\BR} a(x,t,\nabla w) \cdot (\nabla e) \rho \, \d x
\quad \mbox{ for } \ w,e \in  \Wpr.
\end{equation}
By virtue of \eqref{a-mono} and \eqref{a-bdd}, one can easily check that $A^t$ is monotone and continuous in $\Wpr$ (see, e.g.,~\cite[Theorem 1.27]{Roubicek}), and hence, $A^t$ turns out to be maximal monotone in $ \Wpr \times ( \Wpr)^*$ for each $t \in (0,\Sl)$ (see, e.g.,~\cite[Theorem 1.3]{B}).

\subsection{Weak convergence of approximate solutions}\label{Ss:wc}

%Let $(\Omega_n)$ be a sequence of bounded domains in $\Rd$ such that $\Omega_n \subset \Omega_m \subset \Omega$ if $n \leq m$; $\partial \Omega_n$ is smooth, and assume that
%$$
%\Omega_n \to \Omega \ \mbox{ as } \ n \to +\infty,
%$$
%that is, for any $x \in \Omega$, there exists $n \in \N$ such that $x \in \Omega_n$.

Let $(u_n,v_n)$ be an energy solution of \eqref{dnp-n}--\eqref{dnp-ic-n} on $[0,\Sl]$. We first derive from (A1) the following estimates:

\begin{lemma}
\label{Lemma:2.2}
For any $R > 0$ and $0 < t_1 < t_2 < \Sl$, there exists a constant $M > 0$ such that
\begin{align}
\int^{t_2}_{t_1} \int_{\BR} |u_n|^p \, \d x \d t &\leq M,\label{e:unL2}\\
\sup_{t \in [t_1,t_2]} \int_{\BR} |v_n(\cdot,t)|^{p'} \, \d x &\leq M. \label{e:un2(q-1)}
\end{align}
\end{lemma}

Throughout this section, we denote by $M$ a constant which is independent of $n$ and $(x,t)$ but may depend on $R$, $t_1$ and $t_2$ and which may vary from line to line.

\begin{proof}
First, \eqref{e:unL2} follows immediately from (A1) along with Sobolev inequality. 
As for \eqref{e:un2(q-1)}, test \eqref{dnp-n} by $v_n^{p'-1} \zeta_R^{p'}$, where $\zeta_R = \zeta_R(x)$ is a smooth cut-off function satisfying
$$
0 \leq \zeta_R \leq 1 \ \mbox{ in } \Rd, \quad \zeta_R \equiv 1 \ \mbox{ on } \ \BR, \quad \zeta_R \equiv 0 \ \mbox{ on } \ \R^\dim \setminus B_{2R}, \quad |\nabla\zeta_R| \leq \frac C R \ \mbox{ in } \Rd.
$$
Then, for $n$ large enough, we see that
\begin{align*}
\lefteqn{
\frac 1 {p'} \frac \d {\d t} \left( \int_{B_{2R}} |v_n|^{p'} \zeta_R^{p'} \, \d x \right)
+ \int_{B_{2R}} \underbrace{a(x,t,\nabla u_n) \cdot \left(\nabla v_n^{p'-1}\right)}_{\geq \, 0} \zeta_R^{p'} \, \d x
}\\
&= - p' \int_{B_{2R}} a(x,t,\nabla u_n) \cdot (\nabla \zeta_R) \zeta_R^{p'-1} v_n^{p'-1} \, \d x\\
&\leq \frac{C}{R^{p'}} \int_{B_{2R}} |a(x,t,\nabla u_n)|^{p'} \, \d x + \int_{B_{2R}} |v_n|^{p'} \zeta_R^{p'} \, \d x.
\end{align*}
Multiply both sides by $(t-t_1/2)$. Then we deduce that
\begin{align*}
\lefteqn{
\frac 1 {p'} \frac \d {\d t} \left[\biggr(t-\frac{t_1}2\biggr) \int_{B_{2R}} |v_n|^{p'} \zeta_R^{p'} \, \d x \right] 
}\\
&\stackrel{\eqref{a-bdd}}\leq \frac{C}{R^{p'}}\biggr(t-\frac{t_1}2\biggr)
\int_{B_{2R}} \left(|\nabla u_n|^p + k(\cdot,t)\right) \, \d x + \biggr(t-\frac{t_1}2\biggr)\int_{B_{2R}} |v_n|^{p'} \zeta_R^{p'} \, \d x\\
&\quad + \frac 1 {p'} \int_{B_{2R}} |v_n|^{p'} \zeta_R^{p'} \, \d x.
\end{align*} 
Integrating both sides over $(t_1/2,t)$ for $t \in (t_1,t_2)$ and using (A1), we conclude that
$$
\biggr(t-\frac{t_1}2\biggr) \int_{B_{2R}} |v_n(\cdot,t)|^{p'} \zeta_R^{p'} \, \d x \leq M
\quad \mbox{ for } \ t \in (t_1,t_2).
$$
Thus \eqref{e:un2(q-1)} follows. 
\end{proof}

Hence, by (A1) and Lemma~\ref{Lemma:2.2}, we find that, up to a (not relabeled) subsequence,
 \begin{alignat}{4}
 u_n &\to u \quad &&\mbox{ weakly in } L^p(t_1,t_2; W^{1,p}(\BR)),\label{c:unL2H1}\\
 v_n &\to v \quad &&\mbox{ weakly in } L^{p'}(\BR \times (t_1,t_2)),\label{c:vnLp'}\\
 & &&\mbox{ weakly star in } L^\infty(t_1,t_2;L^{p'}(\BR)).\label{c:vnLi}
 \end{alignat}
On the other hand, we observe from \eqref{eq:Q1} that
 \begin{align*}
 \left| \langle A^t (w), e \rangle_{ \Wpr} \right|
 &= \left| \int_{\BR} a(x,t,\nabla w) \cdot (\nabla e) \, \rho \, \d x \right|\\
  &\leq \left( \int_{\BR} |a(x,t,\nabla w)|^{p'} \, \d x \right)^{1/{p'}} \|\nabla e\|_{L^p(\BR)}\|\rho\|_{L^\infty(\BR)}\\
  &\stackrel{\eqref{a-bdd}}\leq C \left( \int_{\BR} \left(|\nabla w|^p+k(\cdot,t)\right) \, \d x \right)^{1/{p'}} \|e\|_{ \Wpr}
 \end{align*}
for $w, e \in  \Wpr$. Thus
 $$
 \|A^t(w)\|_{( \Wpr)^*}^{p'} \leq C \left(\|\nabla w\|_{L^p(\BR)}^p+\|k(\cdot,t)\|_{L^1(\BR)}\right) \quad \mbox{ for } \ w \in  \Wpr.
 $$
 Hence it follows from \eqref{un-bdd} that 
\begin{equation}
\label{eq:Q2}
\mbox{$\eta_n := A^t (u_n)$ is bounded in $L^{p'}(t_1,t_2;( \Wpr)^*)$},
\end{equation} 
and therefore, up to a (not relabeled) subsequence,
\begin{equation}\label{c:eta}
 \eta_n \to \eta \quad \mbox{ weakly in } L^{p'}(t_1,t_2;( \Wpr)^*)
\end{equation}
 for some $\eta \in L^{p'}(t_1,t_2;( \Wpr)^*)$. Moreover, we note that
 \begin{equation}\label{e:HnL2}
 \int^{t_2}_{t_1} \int_{\BR} |a(x,t,\nabla u_n)|^{p'} \, \d x \d t
 \stackrel{\eqref{a-bdd}}\leq C \int^{t_2}_{t_1} \int_{\BR} \left(|\nabla u_n|^p+k(\cdot,t)\right) \, \d x \d t \stackrel{\eqref{un-bdd}}\leq C.
\end{equation}
Hence taking a (not relabeled) subsequence, we get
 \begin{equation}\label{c:HnL2}
 a(\cdot,\cdot,\nabla u_n) \to \xi \quad \mbox{ weakly in } L^{p'}(\BR\times(t_1,t_2))^\dim
 \end{equation}
 for some $\xi \in L^{p'}(\BR\times(t_1,t_2))^\dim$.

Recall \eqref{dnp:wf} for $n > R$ and substitute $w = \overline{\phi \rho} \in W^{1,p}_0(B_n)$ (where $\overline{\phi\rho}$ is the zero extension of $\phi\rho \in W^{1,p}_0(B_R)$ onto $B_n$) for an arbitrary $\phi \in  \Wpr$. Then we see that
$$
\langle \partial_t v_n, \overline{\phi\rho} \rangle_{W^{1,p}_0(B_n)}
 + \int_{\BR} a(x,t,\nabla u_n) \cdot (\nabla \phi) \rho \, \d x
 + \int_{\BR} a(x,t,\nabla u_n) \cdot (\nabla \rho) \phi \, \d x = 0,
$$
which along with \eqref{eq:Q1} yields
\begin{align*}
%\lefteqn{
  \left| \langle \partial_t v_n, \overline{\phi\rho} \rangle_{W^{1,p}_0(B_n)} \right|
%}\\
 &\leq \left\| A^t(u_n) \right\|_{( \Wpr)^*} \|\phi\|_{ \Wpr}\\
 & \quad + \left\| a(\cdot,t,\nabla u_n) \right\|_{L^{p'}(\BR)}
  \left( \int_{\BR} |\phi|^p |\nabla \rho|^p \, \d x \right)^{1/p}\\
 &\leq \left\| A^t(u_n) \right\|_{( \Wpr)^*} \|\phi\|_{ \Wpr} \\
 &\quad + \left\| a(\cdot,t,\nabla u_n) \right\|_{L^{p'}(\BR)} \|\phi\|_{L^p(\BR)} \|\nabla\rho\|_{L^\infty(\BR)}
 \end{align*}
for any $\phi \in  \Wpr$. Hence the arbitrariness of $\phi \in  \Wpr$ yields
$$
\|\rho \partial_t v_n\|_{( \Wpr)^*} 
\leq \left\| A^t(u_n) \right\|_{( \Wpr)^*} + \left\| a(\cdot,t,\nabla u_n) \right\|_{L^{p'}(\BR)} \|\nabla\rho\|_{L^\infty(\BR)},
$$
where $\rho \partial_t v_n \in ( \Wpr)^*$ is defined by 
$$
[\rho \partial_t v_n](\phi) := \langle \partial_t v_n, \overline{\phi\rho}  \rangle_{W^{1,p}_0(B_n)} \lesssim \|\partial_t v_n\|_{W^{-1,p'}(B_n)} \|\phi\|_{ \Wpr}
$$ 
for $\phi \in  \Wpr$. 
Thus, by \eqref{eq:Q2} and \eqref{e:HnL2}, we obtain
$$
 \int^{t_2}_{t_1} \|\rho\partial_t v_n\|_{( \Wpr)^*}^{p'} \, \d t \leq C.
 $$
 Furthermore, one can observe from (A1) that $(\rho v_n)$ is bounded in $L^{p'}(\BR\times(t_1,t_2))$. Therefore noting that $\rho \partial_t v_n = \partial_t [\rho v_n]$ in $( \Wpr)^*$ (see Lemma \ref{a:L:rel} in Appendix for a proof), we deduce that
 \begin{alignat}{4}
  \rho v_n &\to z \quad &&\mbox{ weakly in } L^{p'}(\BR\times(t_1,t_2)),\label{c:rhovnLp'}\\
  \rho\partial_t v_n &\to \partial_t z \quad &&\mbox{ weakly in } L^{p'}(t_1,t_2;( \Wpr)^*),\label{c:db}
 \end{alignat}
for some $z \in L^{p'}(\BR\times(t_1,t_2)) \cap W^{1,p'}(t_1,t_2;( \Wpr)^*)$. 
By virtue of Aubin-Lions-Simon's compactness lemma (see, e.g.,~\cite[Theorem~3]{Simon}) along with the compact embedding,
$$
L^{p'}(\BR) \simeq (L^p(\BR))^* \hookrightarrow (W^{1,p}(\BR))^*
$$
(here we also used the compact embedding $W^{1,p}(\BR) \hookrightarrow L^p(\BR)$ due to the Rellich-Kondrachov theorem), 
by \eqref{c:rhovnLp'} and \eqref{c:db}, we infer that
\begin{equation}\label{c:rhovn}
 \rho v_n \to z \quad \mbox{ strongly in } C([t_1,t_2]; (W^{1,p}(\BR))^*).
\end{equation}
Thus we obtain $z = \rho v$ by \eqref{c:vnLp'}. Moreover, it follows from \eqref{e:un2(q-1)} that
$$
\rho v_n \to \rho v \quad \mbox{ weakly star in } L^\infty(t_1,t_2;L^{p'}(\BR)),
$$
and therefore, $\rho v \in C_{\rm weak}([t_1,t_2];L^{p'}(\BR))$. Hence $v \in C_{\rm weak}([t_1,t_2];L^{p'}(B_{R/2}))$. On the other hand, recalling \eqref{c:unL2H1}, \eqref{c:rhovnLp'}, \eqref{c:rhovn} and applying Proposition \ref{P:Minty} to the maximal monotone operator $u \mapsto \rho \beta(u)$ in $L^p(B_R \times (t_1,t_2)) \times L^{p'}(B_R \times (t_1,t_2))$, we can verify that
\begin{equation}\label{z=rb}
\rho v \in \rho \beta(u) \ \mbox{ a.e.~in } \BR \times (t_1,t_2)
\end{equation}
(and hence, $v \in \beta(u)$ a.e.~in $\BR \times (t_1,t_2)$). Moreover, we obtain
$$
\iint_{\BR \times (t_1,t_2)} v_n u_n \rho \, \d x \d t
\to \iint_{\BR \times (t_1,t_2)} v u \rho \, \d x \d t.
$$
It further follows that
\begin{align*}
\iint_{B_{R}\times(t_1,t_2)} \hat \beta(u_n) \rho\, \d x \d t
&\leq \iint_{B_{R}\times(t_1,t_2)} \hat \beta(u) \rho\, \d x \d t+ \iint_{B_{R}\times(t_1,t_2)} v_n (u_n-u) \rho\, \d x \d t\\
&\to \iint_{B_{R}\times(t_1,t_2)} \hat \beta(u) \rho\, \d x \d t,
\end{align*}
which along with the weak lower semicontinuity of convex functionals yields
\begin{equation}\label{c:int_hb}
\iint_{B_{R}\times(t_1,t_2)} \hat \beta(u_n) \rho\, \d x \d t
\to \iint_{B_{R}\times(t_1,t_2)} \hat \beta(u) \rho\, \d x \d t.
\end{equation}
Hence the strict convexity of $\hat \beta$ (see \eqref{beta-conv}) implies that
\begin{alignat}{4}
u_n &\to u \quad &&\mbox{ strongly in } L^1(B_{R} \times (t_1,t_2)),\label{c:uL1}\\
\hat\beta(u_n)\rho &\to \hat\beta(u)\rho \quad &&\mbox{ strongly in } L^1(B_{R} \times (t_1,t_2))\label{c:hb}
\end{alignat}
(see~\cite[Theorem~3]{Vi}). Moreover, one can also take a (not relabeled) subsequence of $(n)$ such that
\begin{alignat}{4}
u_n(x,t) &\to u(x,t) \quad &&\mbox{ for a.e. } (x,t) \in \BR \times (t_1,t_2),\label{c:un-pt}\\
\hat\beta(u_n(\cdot,t))\rho &\to \hat\beta(u(\cdot,t))\rho \quad &&\mbox{ strongly in $L^1(\BR)$ \ for a.e. } t \in (t_1,t_2).\nonumber
\end{alignat}
Now, note that $(u_n)$ is bounded in $L^\delta(t_1,t_2;L^1(B_R)) \cap L^p(t_1,t_2;W^{1,p}(B_R))$ due to \eqref{un-bdd} and \eqref{e:unL2}. Taking $r \in (p,p^*)$, we deduce that
$$
\|u_n(t)\|_{L^r(\BR)} \leq \|u_n(t)\|_{L^1(\BR)}^{\theta} \|u_n(t)\|_{L^{p^*}(\BR)}^{1-\theta},
$$
where $\theta \in (0,1)$ is given by $1/r = \theta + (1-\theta)/p^*$, and hence,
\begin{align*}
%\int^{t_2}_{t_1} \|u_n(t)\|_{L^r(B_R)}^{\frac p {1-\theta}} \, \d t 
%&\leq \int^{t_2}_{t_1} \|u_n(t)\|_{L^1(B_{R})}^{\frac{p\theta}{1-\theta}} \|u_n(t)\|_{L^{p^*}(B_{R})}^p \, \d t\\
%&\leq C \left( \sup_{t \in (t_1,t_2)} \|u_n(t)\|_{L^1(B_{R})}^{\frac{p\theta}{1-\theta}} \right) \int^{t_2}_{t_1} \|u_n(t)\|_{W^{1,p}(B_{R})}^p \, \d t \leq M,
\|u_n\|_{L^\omega(t_1,t_2;L^r(\BR))}
\leq \|u_n\|_{L^{\delta}(t_1,t_2;L^1(\BR))}^\theta \|u_n\|_{L^p(t_1,t_2;L^{p^*}(\BR))}^{1-\theta} \leq C,
\end{align*}
where $\omega \in (p,+\infty)$ is a constant satisfying $1/\omega = \theta/\delta + (1-\theta)/p$ ($< 1/p$ since $\delta > p$ in (A1)). Hence we can deduce from \eqref{c:uL1} that
\begin{equation}\label{c:un-Lp}
u_n \to u \quad \mbox{ strongly in } L^p(B_R \times (t_1,t_2)).
\end{equation}
Moreover, we also find that, up to a (not relabeled) subsequence,
$$
u_n(t) \to u(t) \quad \mbox{ strongly in } L^p(B_R) \ \mbox{ for a.e. } t \in (t_1,t_2).
$$
Furthermore, exploiting \eqref{c:rhovn} along with \eqref{e:un2(q-1)} and the density of $W^{1,p}(B_R)$ in $L^p(B_R)$, %along with the arbitrariness of $t_1,t_2$ and extracting a (not relabeled) subsequence, 
we can also verify that
$$
\rho v_n(t) \to \rho v(t) \quad \mbox{ weakly in } L^{p'}(B_R) \quad \mbox{ for a.e. } t \in (t_1,t_2)
$$
without taking any further subsequence. Therefore, from  the Fenchel-Moreau identity for the convex conjugate $(\hat \beta)^*$ of $\hat \beta$, it follows that
\begin{align}
\int_{\BR} (\hat\beta)^*(v_n(\cdot,t))\rho \, \d x &= \int_{\BR} v_n(\cdot,t) u_n(\cdot,t)\rho \, \d x - \int_{\BR} \hat\beta(u_n(\cdot,t))\rho \, \d x\nonumber\\
&\to \int_{\BR} v(\cdot,t)u(\cdot,t)\rho \, \d x- \int_{\BR} \hat{\beta}(u(\cdot,t)) \rho \, \d x\nonumber\\
&= \int_{\BR} (\hat\beta)^*(v(\cdot,t))\rho \, \d x\label{c:hb*}
\end{align}
for a.e.~$t \in (t_1,t_2)$. In what follows, we shall denote by $I$ the set of $t \in [t_1,t_2]$ at which \eqref{c:hb*} holds and $u(t) \in W^{1,p}(\BR)$. 

\subsection{Identification of weak limits of nonlinear terms}

We first identify the weak limit $\eta$ of $A^t(u_n)$. Let $\tI, \tJ \in I$ be fixed. Then we have
 \begin{align}
  \lefteqn{
  \int^{\tJ}_{\tI} \langle A^t(u_n) , u_n \rangle_{ \Wpr} \, \d t
  }\nonumber\\
  &= \int^{\tJ}_{\tI} \int_{\BR} a(x,t,\nabla u_n) \cdot (\nabla u_n) \rho \, \d x \d t\nonumber\\
  &\stackrel{\eqref{dnp:wf}}= - \int_{\BR} (\hat\beta)^*(v_n(\cdot,\tJ)) \rho\, \d x
  + \int_{\BR} (\hat\beta)^*(v_n(\cdot,\tI)) \rho\, \d x\nonumber\\
  & \quad - \int^{\tJ}_{\tI} \int_{\BR} a(x,t,\nabla u_n) \cdot (\nabla \rho) u_n \, \d x \d t.\label{dphi*u}
 \end{align}
Here we used the fact that
\begin{align}
\langle \rho\partial_t v_n(t) , u_n(t) \rangle_{ \Wpr} 
&=\langle \partial_t v_n(t) , \rho u_n(t) \rangle_{W^{1,p}_0(B_n)}\nonumber\\
&= \dfrac{\d}{\d t} \int_{B_n} (\hat\beta)^*(v_n(\cdot,t)) \, \rho\, \d x\nonumber\\
& = \dfrac{\d}{\d t} \int_{\BR} (\hat\beta)^*(v_n(\cdot,t)) \,\rho\, \d x, \label{cr}
\end{align}
which is derived from a chain-rule for subdifferentials along with the relation $\beta^{-1} = \partial (\hat\beta)^*$ (i.e., $u_n \in \partial (\hat\beta)^*(v_n)$) as well as the regularity, $v_n \in W^{1,p'}(0,\Sl;W^{-1,p'}(B_n)) \cap C_{\rm weak}([0,\Sl];L^{p'}(B_n))$  and $\rho u_n \in L^p(0,\Sl;W^{1,p}_0(B_n))$ (cf.~Lemma \ref{L:chain-rule} below).  Therefore combining all these facts with \eqref{c:HnL2}, \eqref{c:un-Lp} and \eqref{c:hb*}, we infer that
 \begin{align*}
  \lefteqn{
  \lim_{n \to \infty} \int^{\tJ}_{\tI} \langle A^t(u_n) , u_n \rangle_{ \Wpr} \, \d t
  }\\
  &\stackrel{\eqref{dphi*u}}= - \int_{\BR} (\hat \beta)^*(v(\cdot,\tJ)) \rho\, \d x
  + \int_{\BR} (\hat \beta)^*(v(\cdot,\tI)) \rho\, \d x\\
  &\quad - \int^{\tJ}_{\tI} \int_{\BR} \xi \cdot (\nabla \rho) u \, \d x \d t.
 \end{align*}

 Next, we claim that
 \begin{equation}\label{cl}
 \mbox{(the right-hand side)} = \int^{\tJ}_{\tI} \langle \eta, u \rangle_{ \Wpr} \, \d t.
 \end{equation}
 To see this, we recall the weak form \eqref{dnp:wf} again and observe that
 \begin{equation}\label{aprx-weak}
 \int^{\tJ}_{\tI} \langle \partial_t v_n, \psi \rangle_{W^{1,p}_0(B_n)} \d t
 + \int^{\tJ}_{\tI} \int_{B_n} a(x,t,\nabla u_n) \cdot \nabla \psi \, \d x \d t = 0
 \end{equation}
 for any $\psi \in C^\infty_c(B_n\times[\tI,\tJ])$. 
 Let $\phi(x,t) \in C^\infty(\R^\dim\times[\tI,\tJ])$ and put $\psi(x,t) = \phi(x,t) \rho(x)$ for $n > R$. Then we find that
 \begin{align}
 \int^{\tJ}_{\tI} \langle \partial_t [\rho v_n], \phi \rangle_{ \Wpr} \d t
 + \int^{\tJ}_{\tI} \underbrace{\int_{\BR} a(x,t,\nabla u_n) \cdot (\nabla \phi) \rho \, \d x}_{\quad = \, \langle A^t(u_n), \phi \rangle_{W^{1,p}(\BR)}} \d t\nonumber\\
 + \int^{\tJ}_{\tI} \int_{\BR} a(x,t,\nabla u_n) \cdot (\nabla \rho) \phi \, \d x \d t = 0.\label{ap-eq}
 \end{align}
%where the identity \eqref{ap-eq} corresponds to a weak form of \eqref{aux-eq} and the last term of the left-hand side is nothing but the weak form of the perturbation $F$ defined in the outline. 
%
Passing to the limit as $n \to +\infty$, we derive from \eqref{c:eta}, \eqref{c:HnL2} and \eqref{c:db} with $z=\rho v$ that
 \begin{align}\label{weak_form}
  \int^{\tJ}_{\tI} \langle \partial_t [\rho v], \phi \rangle_{ \Wpr} \d t
 + \int^{\tJ}_{\tI} \langle \eta , \phi \rangle_{ \Wpr}\, \d t
 + \int^{\tJ}_{\tI} \int_{\BR} \xi \cdot (\nabla \rho) \phi \, \d x \d t = 0
 \end{align}
 for any $\phi \in C^\infty([\tI,\tJ]\times \R^\dim)$. Substitute $\phi = u$ (by density) to get
 $$
\int^{\tJ}_{\tI} \langle \partial_t [\rho v], u \rangle_{ \Wpr} \, \d x \d t + \int^{\tJ}_{\tI} \langle \eta , u \rangle_{ \Wpr}\, \d t
 + \int^{\tJ}_{\tI} \int_{\BR} \xi \cdot (\nabla \rho) u \, \d x \d t = 0.
 $$
 Therefore \eqref{cl} follows immediately from the next lemma, which will be proved at the end of this subsection.
\begin{lemma}[Chain-rule with a weight]\label{L:chain-rule}
It holds that
\begin{align}
\lefteqn{
\int^{\tJ}_{\tI} \langle \partial_t [\rho v], u \rangle_{ \Wpr} \, \d t
}\nonumber\\
& = \int_{\BR} (\hat \beta)^*(v(\cdot,\tJ)) \rho\, \d x - \int_{\BR} (\hat \beta)^*(v(\cdot,\tI)) \rho\, \d x\label{chain-rule}
\end{align}
for any $R > 0$ and $\tI,\tJ \in I$. 
\end{lemma}

Combining all these facts, we obtain
$$
\lim_{n \to \infty} \int^{\tJ}_{\tI} \langle A^t(u_n) , u_n \rangle_{ \Wpr} \, \d t
= \int^{\tJ}_{\tI} \langle \eta , u \rangle_{ \Wpr} \, \d t.
$$
Thanks to Proposition \ref{P:Minty}\footnote{To be precise, we apply Proposition \ref{P:Minty} to the operator $\mathcal{A} : \mathcal{V}:=L^p(t_1,t_2;\Wpr) \to \mathcal{V}^*$ defined by
$$
\eta \in \mathcal{A}(u) \quad \mbox{ if and only if } \quad \eta(t) \in A^t(u(t)) \ \mbox{ for a.e. } t \in (t_1,t_2)
$$
for $u \in \mathcal{V}$ and $\eta \in \mathcal{V}^*$. The maximal monotonicity of $\mathcal{A}$ in $\mathcal{V} \times \mathcal{V}^*$ can also be checked as in the case of $A^t$.} we obtain $\eta = A^t(u)$, and moreover, it follows that
 \begin{align}\label{PO}
  \int^{\tJ}_{\tI} \langle \eta, \phi \rangle_{ \Wpr} \, \d t
  = \int^{\tJ}_{\tI} \int_{\BR} a(x,t,\nabla u) \cdot (\nabla \phi) \, \rho \, \d x \d t
 \end{align}
 for any $\phi \in C^\infty(\Rd\times[\tI,\tJ])$. 

For later use, let us next identify the weak limit $\xi$ of $a(\cdot,\cdot,\nabla u_n)$. Note that
$$
\int^{\tJ}_{\tI} \int_{\BR} a(x,t,\nabla u_n) \cdot (\nabla u) \rho \, \d x \d t
\to \int^{\tJ}_{\tI} \int_{\BR} \xi \cdot (\nabla u) \rho \, \d x \d t
$$
and
\begin{align*}
\int^{\tJ}_{\tI} \int_{\BR} a(x,t,\nabla u_n) \cdot (\nabla u) \rho \, \d x \d t
&= \int^{\tJ}_{\tI} \left\langle A^t(u_n), u \right\rangle_{ \Wpr} \, \d t\\
&\to \int^{\tJ}_{\tI} \left\langle \eta, u \right\rangle_{ \Wpr} \, \d t\\
&= \int^{\tJ}_{\tI} \int_{\BR} a(x,t,\nabla u) \cdot (\nabla u) \rho \, \d x \d t.
\end{align*}
Thus
\begin{equation}\label{rel}
\int^{\tJ}_{\tI} \int_{\BR} \xi \cdot (\nabla u) \rho \, \d x \d t
= \int^{\tJ}_{\tI} \int_{\BR} a(x,t,\nabla u) \cdot (\nabla u) \rho \, \d x \d t.
\end{equation}
Since the map $\xi \in L^p(\BR)^\dim \mapsto a(\cdot,t,\xi(\cdot)) \rho \in L^{p'}(\BR)^\dim$ is also maximal monotone in $L^p(\BR)^\dim \times L^{p'}(\BR)^\dim$ for a.e.~$t \in (0,\Sl)$ as in the case of $A^t$,  noting that
\begin{align*}
\int^{\tJ}_{\tI} \int_{\BR} a(x,t,\nabla u_n) \cdot (\nabla u_n) \rho \, \d x \d t
&= \int^{\tJ}_{\tI} \left\langle A^t(u_n), u_n \right\rangle_{ \Wpr} \, \d t\\
&\to \int^{\tJ}_{\tI} \left\langle \eta, u \right\rangle_{\Wpr} \, \d t\\
&= \int^{\tJ}_{\tI} \int_{\BR} a(x,t,\nabla u) \cdot (\nabla u) \rho \, \d x \d t\\
&\stackrel{\eqref{rel}}= \int^{\tJ}_{\tI} \int_{\BR} \xi \cdot (\nabla u) \rho \, \d x \d t
\end{align*}
and employing \eqref{c:unL2H1} and \eqref{c:HnL2}, we can conclude by Proposition \ref{P:Minty} that $\xi = a(\cdot,\cdot,\nabla u)$ a.e.~in $\BR \times (\tI,\tJ)$.

Finally, recalling \eqref{weak_form} along with \eqref{PO} and \eqref{rel}, we further find that
\begin{align*}
- \int^{\tJ}_{\tI} \int_{\BR} v \rho \partial_t \phi \, \d x \d t
+ \int_{\BR} v(\cdot,\tJ) \phi(\cdot,\tJ) \rho \, \d x
 - \int_{\BR} v(\cdot,\tI) \phi(\cdot,\tI) \rho \, \d x\\
+ \int^{\tJ}_{\tI} \int_{\BR} a(x,t,\nabla u) \cdot (\nabla \phi) \rho \, \d x \d t  + \int^{\tJ}_{\tI} \int_{\BR} a(x,t,\nabla u) 
  \cdot (\nabla \rho) \phi \, \d x \d t = 0
 \end{align*}
for all $\phi \in C^\infty(\R^\dim\times[\tI,\tJ])$. 
Now, let $\psi$ be a function of class $C^\infty_c(\R^\dim\times[\tI,\tJ])$ and take $R > 0$ such that
 $$
 \mathrm{supp} \, \psi(\cdot,t) \subset \Br \quad \mbox{ for all } \ t \in [\tI,\tJ].
 $$
Moreover, recall that $\rho \equiv 1$ on $\Br$. Substitute $\phi = \psi$ to the relation above and use the fact that $\nabla \rho \equiv 0$ in $\Br$. Then it follows that
 \begin{align}
 - \int^{\tJ}_{\tI} \int_{\Rd} v \partial_t \psi \, \d x \d t
 + \int_{\Rd} v(\cdot,\tJ) \psi(\cdot,\tJ) \, \d x
  - \int_{\Rd} v(\cdot,\tI) \psi(\cdot,\tI) \, \d x\nonumber\\
+ \int^{\tJ}_{\tI} \int_{\Rd} a(x,t,\nabla u) \cdot \nabla \psi \, \d x \d t = 0\label{weak_form_hat}
 \end{align}
for $\tI,\tJ \in I$, and moreover, one can replace $\tI$ and $\tJ$ with $t_1$ and $t_2$, respectively, by using $v \in C_{\rm weak}([t_1,t_2];L^{p'}(\BR))$ along with $|(t_1,t_2)\setminus I|=0$ as well as the absolute continuity of Lebesgue  integral.

To be precise, we denote by $u_{\BR \times (t_1,t_2)}$ and $v_{\BR \times (t_1,t_2)}$ the limits of (subsequences of) $(u_n)$ and $(v_n)$, respectively, such that all the convergences (in particular, \eqref{c:unL2H1}, \eqref{c:vnLi} and \eqref{c:HnL2}) established so far hold true. From the arbitrariness of $R > 0$ and $0 < t_1 < t_2 < \Sl$, (by a diagonal argument) we can also extract a (not relabeled) subsequence of $(n)$ and obtain a pair of measurable\footnote{The measurability of $u$ in $\Rd \times (0,\Sl)$ is straightforward from the pointwise convergence \eqref{c:un-pt-whole}, and that of $v$ can be proved as follows, although no pointwise convergence has been obtained for $(v_n)$:  One can immediately assure the measurability of $v$ on any compact subsets of $\Rd \times (0,\Sl)$ from the weak convergences of $(v_n)$ obtained so far, and hence, $v \chi_n$ is measurable on $\Rd \times (0,\Sl)$, where $\chi_n$ is the characteristic function supported over the compact set $\overline{B_n} \times [1/n,\Sl-1/n]$. Then $(v \chi_n)(x,t) \to v(x,t)$ for a.e.~$(x,t) \in \Rd \times (0,\Sl)$. Thus $v$ turns out to be measurable in $\Rd \times (0,\Sl)$.} functions $(u,v) : \Rd \times (0,\Sl) \to \R^2$  such that 
\begin{alignat}{4}
&u = u_{\BR \times (t_1,t_2)}, \ v = v_{\BR \times (t_1,t_2)} \ &&\mbox{ a.e.~in } \BR \times (t_1,t_2),\\
%u_n &\to u \quad &&\mbox{ weakly in } L^p(\BR \times (t_1,t_2)),\\
&u_n(x,t) \to u(x,t) \quad &&\mbox{ for a.e. } (x,t) \in \Rd \times (0,\Sl),
\label{c:un-pt-whole}
\end{alignat}
for any $R > 0$ and $0 < t_1 < t_2 < \Sl$ and \eqref{c:unL2H1}, \eqref{c:vnLi} and \eqref{c:HnL2} still hold true with the measurable functions $u,v$ defined over $\Rd \times (0,\Sl)$ instead of $u_{\BR \times (t_1,t_2)},v_{\BR \times (t_1,t_2)}$ for any fixed $R$ and $t_1,t_2$. Then the pair $(u,v) : \Rd \times (0,\Sl) \to \R^2$ satisfies \eqref{weak_form_hat} for any $0 < \tI < \tJ < \Sl$ and $R > 0$. Furthermore, from the arbitrariness of $R>0$, we obtain

\begin{lemma}[Identification of weak limits]\label{L:ident} 
Let $(u_n,v_n)$ be energy solutions to \eqref{dnp-n}--\eqref{dnp-ic-n} on $[0,\Sl]$ such that the assumptions {\rm (A0)} and {\rm (A1)} hold. Then there exists a pair of measurable functions $(u,v) : \Rd \times (0,\Sl) \to \R^2$ such that, up to a subsequence, \eqref{tloc:1}--\eqref{tloc:4} hold true and
\begin{align*}
- \int^{t_2}_{t_1} \int_{\Rd} v \partial_t \psi \, \d x \d t
 + \int_{\Rd} v(\cdot,t_2) \psi(\cdot,t_2) \, \d x
  - \int_{\Rd} v(\cdot,t_1) \psi(\cdot,t_1) \, \d x\nonumber\\
+ \int^{t_2}_{t_1} \int_{\Rd} a(x,t,\nabla u) \cdot \nabla \psi \, \d x \d t = 0
\end{align*}
 for any $\psi \in C^\infty_c([t_1,t_2]\times \R^\dim)$ and for all $0 < t_1 < t_2 < \Sl$, and moreover,
$$
v \in \beta(u) \quad \mbox{ a.e. in } \ \Rd \times (0,\Sl).
$$
\end{lemma}

We close this section with a proof of Lemma \ref{L:chain-rule}.
\begin{proof}[Proof of Lemma \ref{L:chain-rule}]
Define a functional $\phi : W^{1,p}(B_R) \to (-\infty,+\infty]$ by
$$
\phi(w) = \begin{cases}
	   \int_{B_R} \hat\beta(w(\cdot)) \rho \, \d x &\mbox{ if } \ \hat\beta(w(\cdot)) \rho \in L^1(B_R),\\
	   +\infty &\mbox{ otherwise},
	  \end{cases}
$$
for $w \in W^{1,p}(B_R)$. Then $\phi$ is convex and lower semicontinuous on $W^{1,p}(B_R)$, and hence, so is the convex conjugate $\phi^* : (W^{1,p}(B_R))^* \to (-\infty,+\infty]$ defined by $\phi^*(w)=\sup\{\langle w, \zeta \rangle_{W^{1,p}(\BR)} - \phi(\zeta) \colon \zeta \in W^{1,p}(\BR)\}$ for $w \in (W^{1,p}(\BR))^*$. Moreover, recall that $\rho v$ and $u$ lie on $W^{1,p'}(\tI,\tJ;(\Wpr)^*) \cap L^{p'}(\BR \times (\tI,\tJ))$ and $L^p(\tI,\tJ;\Wpr)$, respectively, and that $v \in \beta(u)$ a.e.~in $\BR \times (\hat t_1,\hat t_2)$; therefore, $\rho v \in \partial \phi(u)$, i.e., $u \in \partial \phi^*(\rho v)$. Hence a standard chain-rule formula for subdifferentials yields
$$
\dfrac{\d}{\d t} \phi^*(\rho v(t)) = \langle \partial_t [\rho v](t) , u(t) \rangle_{W^{1,p}(B_R)} \quad \mbox{ for a.e. } t \in (\hat t_1,\hat t_2).
$$
Now, define a functional $\psi : L^{p'}(\BR) \to [0,+\infty]$ by
$$
\psi(w) = \begin{cases}
	   \int_{\BR} (\hat \beta)^*(\rho^{-1}w) \rho \, \d x
	   &\mbox{ if } \ (\hat \beta)^*(\rho^{-1}w) \rho \in L^1(\BR),\\
	   +\infty &\mbox{ otherwise},
	  \end{cases}
$$
for $w \in L^{p'}(\BR)$. Then we claim that
\begin{equation}
\phi^*(w) = \psi(w) \quad \mbox{ if } \ w \in L^{p'}(\BR), \ w \in \beta(z) \rho \ \mbox{ and } \ z \in W^{1,p}(\BR).
\end{equation}
Indeed, we first observe that
$$
\psi(w) = \sup_{\zeta \in L^p(\BR)} \int_{\BR} \left[ \rho^{-1} w\zeta - \hat \beta(\zeta(\cdot)) \right] \rho \, \d x \geq \phi^*(w).
$$ 
On the other hand, the supremum above is attained at $\zeta = z \in \Wpr$ due to the Fenchel-Moreau identity. Thus $\psi(w) = \phi^*(w)$. In particular, we find that
$$
\phi^*(\rho v(t)) = \psi(\rho v(t)) \quad \mbox{ for all } \ t \in I. 
$$
Therefore we obtain
$$
\psi(\rho v(\hat t_2)) - \psi(\rho v(\hat t_1)) = \int^{\hat t_2}_{\hat t_1} \langle \partial_t [\rho v] , u \rangle_{W^{1,p}(B_R)} \, \d t,
$$
which completes the proof.
\end{proof}

\begin{comment}

$$
u \in \beta^{-1}(\rho^{-1}[\rho v]) \quad \mbox{ and } \quad \beta^{-1} = \partial (\hat\beta)^*.
$$
%%
%%
%%
Then one can check the convexity and lower semicontinuity of $\psi$ in $L^{p'}(\BR)$ (by Fatou's lemma for the latter). Furthermore, as in~\cite[p.61]{B}, we can verify that, for $w \in L^{p'}(\BR)$ and $f \in L^p(\BR)$, it holds that
$$
[w,f] \in \partial \psi \quad \mbox{ if and only if } \quad f(x) \in \beta^{-1}(\rho(x)^{-1}w(x)) \ \mbox{ for a.e. } x \in \BR.
$$
We find that $\rho v$ belongs to the effective domain of $\psi$ by noting that 
\begin{align*}
(\hat \beta)^*(\rho^{-1} [\rho v]) \rho = (\hat \beta)^*(v) \rho
= \left[ uv - \hat \beta(u) \right] \rho \in L^1(\BR).
\end{align*}
Hence applying a chain-rule for subdifferentials to $\psi$, we can assure that the function $t \mapsto \psi(\rho v(t))$ is absolutely continuous on $[\tI,\tJ]$, and moreover,
$$
\dfrac{\d}{\d t} \psi(\rho v) = \left\langle \partial_t [\rho v], u \right\rangle_{\Wpr} \quad \mbox{ for a.e. } \ t \in (\tI,\tJ),
$$
which implies \eqref{chain-rule} by integration of both sides over $(\tI,\tJ)$.

\end{comment}

\subsection{Weak formulation including initial data}\label{S:ic}

We next check the weak form \eqref{eq}.
\begin{lemma}\label{L:ic}
In addition to {\rm (A0)} and {\rm (A1)}, assume that {\rm (A2)} is fulfilled. Let $(u,v)$ be the pair of measurable functions defined on $\Rd \times (0,\Sl)$ constructed by Lemma \ref{L:ident}. Then $v$ and $|a(\cdot,\cdot,\nabla u)|$ belong to $L^1(\BR \times (0,T))$ for any $R > 0$ and $T \in (0,\Sl)$, and moreover, it holds that
\begin{align}
- \int^t_0 \int_{\Rd} v \partial_t \psi \, \d x \d \tau
 + \int_{\Rd} v(\cdot,t) \psi(\cdot,t) \, \d x
  - \int_{\Rd} \psi(\cdot,0) \, \d \mu(x)\nonumber\\
+ \int^t_0 \int_{\Rd} a(x,\tau,\nabla u)  \cdot \nabla \psi \, \d x \d \tau = 0\label{wic}
\end{align}
 for any $\psi \in C^\infty_c([0,\Sl)\times \R^\dim)$ and $0 < t < \Sl$.
\end{lemma}

To prove this lemma, we first claim that
\begin{lemma}
For any $R > 0$ and $T \in (0,\Sl)$, $v$ and $|a(\cdot,\cdot,\nabla u)|$ belong to $L^1(\BR\times(0,T))$.
\end{lemma}

\begin{proof}
In what follows, denote by $(u_n,v_n)$ a subsequence of the energy solutions to \eqref{dnp-n}--\eqref{dnp-ic-n} on $[0,\Sl]$ satisfying all the convergences obtained in the last subsection. Fix $R > 0$ and $T \in (0,\Sl)$. Based on Lemma \ref{L:ident} and its proof in the last subsection, we can assure that, for each $\vep \in (0,T)$, there exists a limit $(u_\vep,v_\vep) : \BR \times (\vep,T) \to \R^2$ such that
\begin{alignat*}{3}
u_n &\to u_\vep \quad &&\mbox{ weakly in } L^p(\vep,T; W^{1,p}(\BR)),%\hookrightarrow L^1(\vep,T; W^{1,1}(\BR)),
\\
v_n &\to v_\vep \quad &&\mbox{ weakly in } L^{p'}(\vep,T; L^{p'}(\BR)),%\hookrightarrow L^1(\vep,T; L^1(\BR)).
\\
a(\cdot,\cdot,\nabla u_n) &\to a(\cdot,\cdot,\nabla u_\vep) \quad &&\mbox{ weakly in } L^{p'}(\vep,T;L^{p'}(\BR))^N,\\
u_\vep &= u, \quad v_\vep = v \quad &&\mbox{ a.e.~in } \ \BR \times (\vep,T),
\end{alignat*}
where $u,v : \Rd \times (0,\Sl) \to \R^2$ are measurable functions constructed in the last subsection (see Lemma \ref{L:ident}). By virtue of (A2) along with the (weak) lower semicontinuity of norms, it follows that
\begin{align}
 \int^t_\vep \int_{\BR} \left( |v_\vep| + |a(x,\tau,\nabla u_\vep)| \right) \, \d x \d \tau 
&\leq \liminf_{n \to +\infty} \int^t_\vep \int_{\BR} \left( |v_n| + |a(x,\tau,\nabla u_n)| \right) \, \d x \d \tau \nonumber\\
&\leq \sup_{n \in \N} \left( \int^t_0 \int_{\BR} \left( |v_n| + |a(x,\tau,\nabla u_n)| \right) \, \d x \d \tau \right)\nonumber \\
&\stackrel{\text{(A2)}}{\to} 0 \quad \mbox{ as } \ 0 < \vep < t \to 0_+.\label{e:L1}
\end{align}
Here and subsequently, we shall use the same notation for the zero extensions of $u_\vep$, $a(\cdot,\cdot,\nabla u_\vep)$ and $v_\vep$ onto $\BR \times (0,T)$ for each $\vep \in (0,T)$. We claim that $(v_\vep)$ and $(a(\cdot,\cdot,\nabla u_\vep)_j)$ form Cauchy sequences in $L^1(B_R \times (0,T))$ for $j = 1,2,\ldots,N$. Indeed, we see that, for any $0 < \vep' < \vep < T$,
 \begin{align*}
  \int^T_0 \int_{\BR} \left(|v_\vep - v_{\vep'}|+|a(x,t,\nabla u_\vep) - a(x,t,\nabla u_{\vep'})|\right) \, \d x \d t
\\
= \int^{\vep}_{\vep'} \int_{\BR} \left(|v_{\vep'}|+|a(x,t,\nabla u_{\vep'})|\right) \, \d x \d t,
 \end{align*}
and moreover, the right-hand side converges to zero as $\vep \to 0_+$ by \eqref{e:L1}. Therefore we deduce that
 $$
 v_\vep \to v \ \mbox{ and } \ a(\cdot,\cdot,\nabla u_\vep)_j \to a(\cdot,\cdot,\nabla u)_j \ \mbox{ strongly in } L^1(\BR\times(0,T))
 $$
for $j = 1,2,\ldots,N$. Here the limits are identified by means of the fact that $v_\vep = v$ and $a(\cdot,\cdot,\nabla u_\vep)=a(\cdot,\cdot,\nabla u)$ a.e.~in $\BR\times (\vep,T)$. In particular, $v$ and $|a(\cdot,\cdot,\nabla u)|$ belong to $L^1( \BR \times (0,T) )$ for any $R > 0$ and $T \in (0,\Sl)$.
The proof is complete.
\end{proof}

 Now, we are in a position to prove

\begin{lemma}
The weak form \eqref{wic} holds true.
\end{lemma}

\begin{proof}
 Let $\psi \in C^\infty_c([0,\Sl)\times \R^\dim)$ and $t \in (0, \Sl)$ be fixed.
 Take $R > 0$ such that $\mathrm{supp} \, \psi(\tau,\cdot) \subset \BR$ for any $\tau \in [0,t]$. For $n > R$, we observe that
  \begin{align*}
   \lefteqn{
  0 \stackrel{\eqref{dnp:wf}}= \int^{t}_{0} \langle \partial_t v_n, \psi \rangle_{W^{1,p}_0(B_n)} \, \d \tau
  + \int^{t}_{0} \int_{\Rd} a(x,\tau,\nabla u_n) \cdot \nabla \psi \, \d x \d \tau}\\
   &= \int^t_{0} \langle \partial_t v_n, \psi \rangle_{W^{1,p}_0(B_n)} \, \d \tau
  + \int^{\vep}_{0} \int_{\Rd} a(x,\tau,\nabla u_n) \cdot \nabla \psi \, \d x \d \tau\\
  &\quad + \int^{t}_{\vep} \int_{\Rd} a(x,\tau,\nabla u_n) \cdot \nabla \psi \, \d x \d \tau\\
  &= - \int^{\vep}_{0} \int_{\Rd} v_n \partial_t \psi \, \d x \d \tau
   - \int^t_{\vep} \int_{\Rd} v_n \partial_t \psi \, \d x \d \tau\\
   &\quad + \int_{\Rd} v_n(\cdot,t) \psi(\cdot,t) \, \d x
  - \int_{\Rd} \mu_n \psi(\cdot,0)\, \d x \\
   & \quad + \int^{\vep}_{0} \int_{\Rd} a(x,\tau,\nabla u_n) \cdot \nabla \psi \, \d x \d \tau
   + \int^{t}_{\vep} \int_{\Rd} a(x,\tau,\nabla u_n) \cdot \nabla \psi \, \d x \d \tau.
  \end{align*}
  Set
 $$
  I_{\vep,n} := - \int^{\vep}_{0} \int_{\Rd} v_n \partial_t \psi \, \d x \d \tau
   + \int^{\vep}_{0} \int_{\Rd} a(x,\tau,\nabla u_n) \cdot \nabla \psi \, \d x \d \tau
 $$
and repeat the same argument as in the proof of Lemma \ref{L:ident} with $t_1=\vep$ and $t_2=t$. Then we deduce from \eqref{ini-hypo} that
\begin{align*}
 \lim_{n \to \infty} I_{\vep,n}
 - \int^{t}_{\vep} \int_{\Rd} v \partial_t \psi \, \d x \d \tau
 + \int_{\Rd} v(\cdot,t) \psi(\cdot,t) \, \d x
  - \int_{\Rd} \psi(\cdot,0) \, \d \mu(x)\\
+ \int^{t}_{\vep} \int_{\Rd} a(x,\tau,\nabla u) \cdot \nabla \psi \, \d x \d \tau = 0.
\end{align*}
Here we used \eqref{c:HnL2} along with the relation $\xi = a(\cdot,\cdot,\nabla u)$. By (A2), we find that
$$
c_\vep := \lim_{n \to \infty} I_{\vep,n} \to 0 \quad \mbox{ as } \ \vep \to 0_+.
$$
Thus we conclude that
\begin{align}
  c_\vep - \int^{t}_\vep \int_{\Rd} v \partial_t \psi \, \d x \d \tau
 + \int_{\Rd} v(\cdot,t) \psi(\cdot,t) \, \d x
  - \int_{\Rd} \psi(\cdot,0) \, \d \mu(x) \nonumber\\
+ \int^{t}_\vep \int_{\Rd} a(x,\tau,\nabla u) \cdot \nabla \psi \, \d x \d \tau = 0.\label{ep}
\end{align}
Passing to the limit as $\vep \to 0_+$, since $v$ and $|a(\cdot,\cdot,\nabla u)|$ are integrable over $\BR \times (0,t)$ and so is $v(\cdot, t)$ over $\BR$, we conclude that
 \begin{align*}
 - \int^{t}_0 \int_{\Rd} v \partial_t \psi \, \d x \d \tau
 + \int_{\Rd} v(\cdot,t) \psi(\cdot,t) \, \d x
  - \int_{\Rd} \psi(\cdot,0) \, \d \mu(x) \nonumber\\
+ \int^{t}_0 \int_{\Rd} a(x,\tau,\nabla u) \cdot \nabla \psi \, \d x \d \tau = 0
\end{align*}
due to the absolute continuity of the Lebesgue integral. Therefore combining all the facts proved so far, we conclude that $(u,v)$ is a local-energy solution of the Cauchy problem \eqref{dnp}, \eqref{dnp-ic} in $(0,\Sl)$ in the sense of Definition \ref{D:sol-dnp}.
\end{proof}

Thus we have proved Theorem \ref{T:loc}.

\section{Finsler FDE}\label{S:FFDE}

In this section, we shall apply Theorem \ref{T:loc} to construct a local-energy solution of the Cauchy problem for the Finsler FDE, 
that is, \eqref{pde}, \eqref{ic} with $2 < q < +\infty$. In the rest of the paper, we shall write
$$
\B_R = \left\{ x \in \Rd \colon H_0(x) < R \right\} \quad \mbox{ for } \ R > 0,
$$
where $H_0(\cdot)$ stands for the dual norm of $H(\cdot)$ (see \S \ref{Ss:Finsler}). To apply Theorem \ref{T:loc}, we set
\begin{equation}\label{F-set}
p = 2, \quad \beta(u) = |u|^{q-2}u, \quad a(x,t,\xi) = H(\xi) \nabla_\xi H(\xi)%\quad \hat a(x,t,\xi) = \frac 1 2 H(\xi)^2 
\end{equation}
and take $H_0(\cdot)$ as a norm of $\Rd$ (i.e., set $\BR = \B_R$). Then \eqref{pde}, \eqref{ic} is reduced to \eqref{dnp}, \eqref{dnp-ic}. Moreover, one can check \eqref{a-mono} and \eqref{a-bdd} by noting that
\begin{gather*}
%\hat a(x,t,\xi) \gtrsim |\xi|^2 \quad \mbox{ and } \quad
\left( a(x,t,\xi)-a(x,t,\eta) \right) \cdot (\xi - \eta) = \big( H(\xi) \nabla_\xi H(\xi)-H(\eta)\nabla_\xi H(\eta) \big) \cdot (\xi - \eta) \geq 0,\\
%|a(x,t,\xi)| \lesssim 
H_0(a(x,t,\xi)) = H_0(H(\xi)\nabla_\xi H(\xi)) = H(\xi), %\lesssim |\xi|
\end{gather*}
for $x, \xi, \eta \in \Rd$ and $t > 0$. Moreover, we set
$$
\hat\beta(u) = \frac 1 q |u|^q,
$$
which is strictly convex by $q > 1$. Thus \eqref{beta-conv} follows.

\subsection{Approximation}\label{Ss:aprx}

This part is common for both fast diffusion ($2 < q < +\infty$) and porous medium ($1 < q < 2$) cases. Let us consider the following approximate problems posed on the balls $\B_n$,
\begin{alignat}{4}
\partial_t u_n^{q-1} &= \Lap u_n \ &&\mbox{ in } \B_n \times (0,+\infty),\label{pde-n}\\
u_n &= 0 \ &&\mbox{ on } \partial \B_n \times (0,+\infty),\label{bc-n}\\
u_n^{q-1} &= \mu_n \ &&\mbox{ in } \B_n \times \{0\},\label{ic-n}
\end{alignat}
where $u_n^{q-1} := |u_n|^{q-2}u_n$ and $(\mu_n)$ is a sequence in $C^\infty_c(\B_n)$ such that
\begin{equation}\label{ini-hypo-F}
\mu_n \to \mu \quad \mbox{ weakly star in } \mathcal M(\R^\dim).
\end{equation}
Thanks to~\cite{B-dn}, for each $n \in \N$, one can construct an energy solution $u_n$ of \eqref{pde-n}--\eqref{ic-n} such that
\begin{align}
 u_n &\in C_{\rm weak}([0,T];H^1_0(\B_n)) \cap C([0,T];L^q(\B_n)),\\
 u_n^{q-1} &\in W^{1,\infty}(0,T;H^{-1}(\B_n)) \cap C_{\rm weak}(0,T;L^2(\B_n)) \cap C([0,T];L^{q'}(\B_n))
\end{align}
for any $T > 0$, and moreover, it holds that
$$
\left\langle \partial_t u_n^{q-1}(t), w \right\rangle_{H^1_0(\B_n)}
+ \int_{\B_n} H(\nabla u_n(x,t)) \nabla_\xi H(\nabla u_n(x,t)) \cdot \nabla w(x) \, \d x = 0
$$
for any $w \in H^1_0(\B_n)$ and a.e.~$t > 0$. Thus (A0) has been ensured.

Furthermore, put $w = (u_n^{q-1} - M_n)_+$ and choose a constant $M_n > \|\mu_n\|_{L^\infty(\B_n)}$. Then we have
$$
\frac 1 2 \frac \d {\d t} \int_{\B_n} (u_n^{q-1}-M_n)_+^2 \, \d x \leq 0
\quad \mbox{ for a.e. } t \in (0,T),
$$
which yields $(u_n(x,t)^{q-1}-M_n)_+ = 0$, that is, $u_n(x,t)^{q-1} \leq M_n$, for a.e.~$(x,t) \in \B_n \times (0,T)$. One can similarly prove that $u_n^{q-1} \geq - M_n$ a.e.~in $\B_n \times (0,T)$. Therefore $u_n$ lies on $L^\infty(\B_n\times(0,T))$ (and hence, $u_n \in C_{\rm weak\,*}([0,T];L^\infty(\B_n))$). Moreover, in particular for the fast diffusion case, since $\mu$ is supposed to be non-negative, one can assume $\mu_n \geq 0$, and hence, so is $u_n$ a.e.~in $\B_n \times (0,T)$.

\subsection{Local-energy estimates for the Finsler FDE}\label{S:FDE}

We next check (A1) and (A2) with \eqref{F-set} and $S = +\infty$. To this end, we shall derive local-energy estimates for $(u_n)$. The following lemma are also valid for \emph{locally bounded non-negative local weak solutions} on general domains $\Omega \subset \Rd$. In the rest of this section, as an independent interest, for any $T > 0$, we shall establish local estimates for a \emph{non-negative measurable function} $u : \Omega \times (0,T) \to [0,+\infty)$ satisfying
\begin{align}\label{es:1}
u \in L^\infty_{\rm loc}(\Omega \times (0,T)) \cap L^2_{\rm loc}(0,T;H^1(B)),\quad
%u^{q-1} &\in C((0,T);L^{q'}(B)) \cap C_{\rm weak}((0,T);L^2(B)),\label{es:2}
u^{q-1} \in W^{1,2}_{\rm loc}(0,T;H^{-1}(B)),
\end{align}
where $H^{-1}(B)$ stands for the dual space of $H^1_0(B)$, for any domain $B \Subset \Omega$ (i.e., $B$ is bounded and $\overline{B} \subset \Omega$) and
\begin{equation}
- \iint_{Q_T} u^{q-1} \partial_t \varphi \, \d x \d t + \iint_{Q_T} H(\nabla u) \nabla_\xi H(\nabla u) \cdot \nabla \varphi \, \d x \d t = 0,\label{es:3}
\end{equation}
where $Q_T := \Omega \times (0,T)$, for all $\varphi \in C^\infty_c(Q_T)$. %satisfying $\varphi \geq 0$ in $Q_T$.
Hence $u^{q-1} \in C_{\rm weak \,*}(0,T;L^\infty(B))$, and moreover, we note from \eqref{es:3} that%, for every bounded domain $B \Subset \Omega$ and $T>0$,
\begin{equation}
\left\langle \partial_t u^{q-1}(t), w \right\rangle_{H^1_0(B)} + \int_B H(\nabla u(t)) \nabla_\xi H(\nabla u(t)) \cdot \nabla w \, \d x = 0
\label{es:3p}
\end{equation}
for a.e.~$t \in (0,T)$ and any $w \in H^1_0(B)$.

In the rest of this subsection, e.g.,  in \eqref{e:lLE} below, the origin of balls can be shifted to suitable $y_0 \in \R^N$. For instance, $\B_R$ can be replaced with $\B_R+y_0$ for any $y_0 \in \R^N$ satisfying $\B_{2R}+y_0 \subset \Omega$. Let us start with a local $L^\infty$-estimate.

\begin{lemma}[Local $L^\infty$ estimates for the Finsler FDE]\label{lLE}
Let $2 < q < +\infty$ and let $r > 1/(q-1)$ be any number satisfying
$$
\kappa_r = 2r - \dim \frac{q-2}{q-1} > 0.
$$
Let $u = u(x,t) : Q_T \to [0,+\infty)$ be a non-negative measurable function satisfying \eqref{es:1}, \eqref{es:3}. Then there exists a constant $C = C(\dim,q,r)$ such that
\begin{equation}\label{e:lLE}
 \|u(\cdot,t)\|_{L^\infty(\B_R)} \leq C t^{-\frac\dim{\kappa_r(q-1)}} \left( \sup_{\frac t8 < \tau < t} \int_{\B_{2R}}  u(x,\tau)^{r(q-1)} \, \d x \right)^{\frac2{\kappa_r(q-1)}} + C \left( \dfrac t {R^2} \right)^{\frac 1 {q-2}}
\end{equation}
for any $t \in (0,T)$ and $R > 0$ satisfying $\B_{2R} \subset \Omega$.
\end{lemma}

\begin{proof}[Outline of proof]
The lemma can be proved by modifying the argument in~\cite[Proof of Theorem 3.1, p.789]{DK92}, where an $L^\infty$ estimate is established for the usual fast diffusion equation. More precisely, we test \eqref{es:3p} with $w=(u-k_n)_+^{r(q-1)-1} \zeta_n^2$, where $(k_n)$ is a sequence increasingly converging to a positive constant $k$ and $(\zeta_n)$ is a sequence of smooth cut-off (in space and time) functions whose supports are rectangles $Q_n \subset Q_T$ decreasing in $n$, and derive energy inequalities (for $w_n := (u-k_n)_+^{r(q-1)/2} \zeta_n$). Choosing $k$ large enough and employing a standard argument based on the H\"older and Gagliard-Nirenberg inequalities as well as an iteration (for a recurrence inequality), one can conclude that $u \leq k$ on a bounded subdomain $Q_\infty$ of $Q_T$. However, due to the choice, $k$ still involves the $L^\infty$ norm $\|u\|_{L^\infty(Q_0)}$ with a power less than $1$ (as well as $\|u\|_{L^{r(q-1)}(Q_0)}$), where $Q_0$ is strictly larger than $Q_\infty$. Therefore running another iteration, we exclude the $L^\infty$ norm from the bound, and thus, the assertion follows. In order to derive energy estimates (for $w_n$), the following algebraic inequality may be useful:
\begin{align*}
\lefteqn{
c H(\Vec a + \Vec b)^2 - (H\nabla_\xi H)(\Vec a + \Vec b) \cdot \Vec b
}\\
&\geq
c H(\Vec a + \Vec b)^2 - H(\Vec a + \Vec b) H_0 \left(\nabla_\xi H(\Vec a + \Vec b) \right) H(\Vec b)\\
&\geq \dfrac{c}{2} H(\Vec a + \Vec b)^2 - \dfrac 1 {2c} H(\Vec b)^2\\
&\geq \dfrac{c}{2} \left( H(\Vec a)^2 + H(\Vec b)^2 \right) - c H(\Vec a) H(\Vec b)- \dfrac 1 {2c} H(\Vec b)^2\\
&\geq \dfrac{c}{4} H(\Vec a)^2 - \frac 12 \left(c + \frac 1 c\right) H(\Vec b)^2
\end{align*}
for $\Vec a, \Vec b \in \Rd$ and $c > 0$ (here we also used \eqref{eq:2.1}), since the norm $H(\cdot)$ is not necessarily induced by an inner product.
\end{proof}

We next establish local $L^r$-estimates ($1 < r < +\infty$) for $u^{q-1}$.

\begin{lemma}[Local $L^r$ estimates for the Finsler FDE]\label{L:lLr}
Let $u = u(x,t) : Q_T \to [0,+\infty)$ be a non-negative measurable function satisfying \eqref{es:1}, \eqref{es:3}. Then for any $r > 1$ there exists a constant $C = C(\dim,q,r) > 0$ such that
$$
\sup_{\tau \in (0,t)}\int_{\B_R} u(x,\tau)^{r(q-1)} \, \d x
\leq C \int_{\B_{2R}} u(x,0_+)^{r(q-1)} \, \d x + 
 C \left( \frac{t^r}{R^{\kappa_r}} \right)^{\frac{q-1}{q-2}},
$$
provided that
\begin{equation}\label{lLr-h}
\int_{\B_{2R}} u(x,0_+)^{r(q-1)} \, \d x := \liminf_{\tau \to 0_+} \int_{\B_{2R}} u(x,\tau)^{r(q-1)} \, \d x < +\infty,
\end{equation}
for any $t \in (0,T)$ and $R > 0$ satisfying $\B_{2R} \subset \Omega$.
\end{lemma}

\begin{proof}[Outline of proof]
The lemma can be proved by following the argument in~\cite[Proof of Theorem 2.2, p.113]{V94}, where a doubly-nonlinear parabolic equation with power nonlinearity is studied. To be more precise, test \eqref{es:3p} with $w=(u^{q-1})^{r-1} \zeta_n^2$, where $(\zeta_n)$ is a sequence of smooth cut-off functions in space only such that $\zeta_n \equiv 1$ on $\B_{R_n}$ and $\mathrm{supp} \, \zeta_n = \overline{\B_{R_{n+1}}}$. Then it follows that
\begin{align*}
\lefteqn{
\frac 1 r \frac \d {\d t} \left( \int_{\B_{R_{n+1}}} u^{r(q-1)} \zeta_n^2 \, \d x \right)
}\\
&\quad + (q-1)(r-1) \int_{\B_{R_{n+1}}} \underbrace{H(\nabla u) \nabla_\xi H(\nabla u) \cdot (\nabla u)}_{= H(\nabla u)^2} u^{(q-1)(r-1)-1} \zeta_n^2 \, \d x\\
&\leq -2 \int_{\B_{R_{n+1}}} H(\nabla u)\nabla_\xi H(\nabla u) \cdot (\nabla \zeta_n) u^{(q-1)(r-1)} \zeta_n \, \d x\\
&\leq \frac{(q-1)(r-1)}2 \int_{\B_{R_{n+1}}} H(\nabla u)^2 u^{(q-1)(r-1)-1} \zeta_n^2 \, \d x\\
&\quad + C_{q,r} \int_{\B_{R_{n+1}}} H(\nabla \zeta_n)^2 u^{(q-1)(r-1)+1} \, \d x.
\end{align*}
Therefore integrating both sides over $(s,t)$ and noting that $(q-1)(r-1)+1 < r(q-1)$ by $q > 2$, we can derive an energy estimate for the $L^r$ norm of $u(\cdot,t)^{q-1}$ over $\B_{R_n}$. Then the bound still involves the $L^r$ norm over $\B_{R_{n+1}}$, which is larger than $\B_{R_n}$, with some power less than $1$ as well as that for the initial datum. However, an iterative argument (similar to the second iteration in the proof of Lemma \ref{lLE}) enables us to exclude it from the bound. Finally, passing to the limit as $s \to 0_+$, we obtain the assertion, provided that \eqref{lLr-h} holds.
\end{proof}

Moreover, we shall establish a local $L^1$ estimate for $u^{q-1}$.

\begin{lemma}[Local $L^1$ estimate for the Finsler FDE]\label{L:lL1}
Let $u=u(x,t) : Q_T \to [0,+\infty]$ be a non-negative measurable function satisfying \eqref{es:1}, \eqref{es:3}. Then there exists a constant $C = C(\dim,q) > 0$ such that
\begin{equation}\label{e:lL1}
\sup_{\tau \in (0,t)}\int_{\B_R} u(x,\tau)^{q-1} \, \d x
\leq C \int_{\B_{2R}} u(x,0_+)^{q-1} \, \d x + 
 C \left( \frac{t}{R^{\kappa_1}} \right)^{\frac{q-1}{q-2}},
\end{equation}
where $\kappa_1$ is given by \eqref{kappa_r} with $r = 1$, provided that
\begin{equation*}
\int_{\B_{2R}} u(x,0_+)^{q-1} \, \d x := \liminf_{\tau \to 0_+} \int_{\B_{2R}} u(x,\tau)^{q-1} \, \d x < +\infty,
\end{equation*}
for any $t \in (0,T)$ and $R > 0$ satisfying $\B_{2R} \subset \Omega$.
\end{lemma}

\begin{proof}[Outline of proof]
The lemma can be proved by following the argument in~\cite[Proof of Theorem 2.3, p.116]{V94}. We may test \eqref{es:3p} with $w=\zeta_n^2$, which corresponds to the test function used in the proof of Lemma \ref{L:lLr} for the choice $r = 1$. Then we have
\begin{align*}
\lefteqn{
\int_{\B_{R_{n+1}}} u(x,t)^{q-1} \zeta_n(x)^2\, \d x
}\nonumber\\
&= \int_{\B_{R_{n+1}}} u(x,s)^{q-1} \zeta_n(x)^2\, \d x
- 2 \int^t_s \int_{\B_{R_{n+1}}} H(\nabla u)\dH(\nabla u) \cdot (\nabla \zeta_n) \zeta_n \, \d x \d \tau%\label{L1-idt}
\end{align*}
for $0 < s < t < T$. However, the gradient term with a ``good'' sign does not appear any longer due to the choice $r=1$, and hence, we need establish a gradient estimate separately to control the last term of the right-hand side. It will be obtained in the following lemma. Then the rest of proof runs as in the proof of Lemma \ref{L:lLr}.
\end{proof}

\prf{
\begin{proof}
Substitute $w = \zeta^2$ to \eqref{es:3p}, where $\zeta = \zeta(x)$ denotes a smooth cut-off function supported over the closure of a bounded open set $B \subset \Omega$ in $\R^\dim$ and will more precisely be specified later. Then
$$
\dfrac{\d}{\d t} \left( \int_B u^{q-1} \zeta^2 \, \d x \right)
\leq - 2\int_B H(\nabla u)\dH(\nabla u) \cdot (\nabla \zeta) \zeta \, \d x.
$$
Integrate both sides over $(s,t)$ in time. We see that
\begin{align}
\lefteqn{
\int_B u(x,t)^{q-1} \zeta(x)^2\, \d x
}\nonumber\\
& \leq \int_B u(x,s)^{q-1} \zeta(x)^2\, \d x
- 2 \int^t_s \int_B H(\nabla u)\dH(\nabla u) \cdot (\nabla \zeta) \zeta \, \d x \d \tau.\label{L1-idt}
\end{align}
Hence in order to derive a desired estimate, it suffices to estimate the last term of the inequality above (to be continued). 
\end{proof}
To this end, we need to set up an auxiliary estimate. In what follows, we write $u_\vep := u + \vep$ for $\vep > 0$. 
}

\begin{lemma}\label{L:lHL1}
Let $B \Subset \Omega$ be a domain in $\Rd$ and let $\zeta : \Rd \to [0,1]$ be a smooth cut-off function supported over the set $\overline{B}$. Let $u=u(x,t) : Q_T \to [0,+\infty]$ be a non-negative measurable function satisfying \eqref{es:1}, \eqref{es:3}. Then there exists a constant $C = C(q)>0$ such that
\begin{align}
\lefteqn{
\int^t_s \int_B \left| H(\nabla u) \dH(\nabla u) \cdot (\nabla \zeta) \right| \zeta \, \d x \d \tau
}\nonumber\\
&\leq C (t-s)^{\frac12} \|H(\nabla \zeta)\|_{L^\infty(B)} |B|^{1-r} \left( \sup_{\tau \in (s,t)} \int_B u_\vep(x,\tau)^{q-1} \, \d x\right)^r,\label{A5}
\end{align}
where $u_\vep := u + \vep$ with the choice $\vep^{q-2}=(t-s)\|H(\nabla \zeta)\|_{L^\infty(B)}^2$ and $r = \frac q {2(q-1)} \in (0,1)$, for any $0 < s < t < T$.
\end{lemma}

\begin{proof}[Outline of proof]
The lemma can be proved as in~\cite[Lemma 3.1, p.114]{V94} (cf.~\cite[Lemma I.2.2]{DiHe89}). We test the equation by $t^\alpha(u_\vep^{q-1})^{r-1} \zeta_n^2 $ with $r \in (0,1)$ and $\alpha > 0$. Then due to $q > 2$ and $r \in (0,1)$, one can derive an energy inequality where the integral of the product between $H(\nabla u_\vep)^2$ and a power of $u_\vep$ appears with a ``good'' sign. Moreover, \eqref{A5} follows from H\"older's inequality as well as the choice of parameters, $r=\frac q{2(q-1)}$ and $\alpha=\frac12$.
\end{proof}

Combining all these lemmas, we can verify

\begin{corollary}\label{C:est}
Let $u = u(x,t) : Q_T \to [0,+\infty)$ be a non-negative measurable function satisfying \eqref{es:1}, \eqref{es:3} such that 
\begin{equation}\label{MR}
\int_{\B_{2R}} u(x,0_+)^{q-1} \, \d x 
:= \liminf_{\tau\to0_+} \int_{\B_{2R}} u(x,\tau)^{q-1} \, \d x \leq M_R
\end{equation}
for some constant $M_R>0$. Then the following {\rm (i)--(iii)} hold\/{\rm :}
\begin{enumerate}
 \item There exists a constant $C = C(N,q) > 0$ such that
\begin{align}
\lefteqn{
\int^t_0 \int_{\B_R} H(\nabla u) \, \d x \d \tau
}\nonumber\\
&\leq C t^{\frac12} R^{\dim(1-r)} \left(\sup_{\tau \in (0,t)} \int_{\B_{2R}} u(x,\tau)^{q-1} \, \d x\right)^r + C t^{\frac{q-1}{q-2}} R^{\dim -\frac q {q-2}},\label{e:H-L1}
\end{align}
where $r := \frac q {2(q-1)} \in (0,1)$, for any $t \in (0,T)$ and $R > 0$ satisfying $\B_{2R} \subset \Omega$. In particular, for any $R >0$ satisfying $\B_{2R} \subset \Omega$, there exists a constant $C > 0$ depending only on $N$, $q$, $R$, $T$ and $M_R$ such that
\begin{equation}\label{e:H-L1-2}
\int^t_0 \int_{\B_R} H(\nabla u) \, \d x \d \tau \leq Ct^{\frac12} 
\end{equation}
for all $t \in (0,T)$.
 \item Suppose that $\kappa_1 > 0$, that is, $q < 2(\dim-1)/(\dim-2)_+$ {\rm (}see~\eqref{hypo-q}{\rm )}. For any $0 < t_1 < t_2 < T$ and $R > 0$ satisfying $\B_{2R} \subset \Omega$, there exists a constant $C > 0$ depending only on $N$, $q$, $R$, $t_1$, $t_2$ and $M_R$ such that
$$
\int^{t_2}_{t_1} \int_{\B_R} H(\nabla u)^2 \, \d x \d \tau \leq C.
$$
 \item Under the same assumption as above, %Suppose that $\kappa_1 > 0$, that is, $q < 2(\dim-1)/(\dim-2)_+$ {\rm (}see~\eqref{hypo-q}{\rm )}. 
for any $0 < t_1 < t_2 < T$, $R > 0$ satisfying $\B_{2R} \subset \Omega$ and $r \in [1,+\infty)$, there exists a constant $C > 0$ depending only on $N$, $q$, $R$, $t_1$, $t_2$, $M_R$ and $r$ such that
$$
\sup_{\tau \in (t_1,t_2)} \int_{\B_R} u(x,\tau)^r \, \d x \leq C.
$$
\end{enumerate}
\end{corollary}

\begin{proof}[Outline of proof]
Concerning (i), \eqref{e:H-L1} can be obtained as a by-product of the proof for Lemma \ref{L:lHL1} (see~\cite[Lemma 3.1, p.114]{V94}). Moreover, combining \eqref{e:H-L1} with Lemma \ref{L:lL1}, one can prove \eqref{e:H-L1-2}. The assertion (ii) can also be verified by recalling the proofs of Lemmas \ref{L:lL1} and \ref{L:lHL1} and by using Lemma \ref{lLE} with $r = 1$ along with Assumption \eqref{hypo-q}, that is, $\kappa_1>0$. Furthermore, under the same assumption above, (iii) follows immediately from Lemma \ref{lLE} with $r = 1$ along with Lemma \ref{L:lL1}.
\end{proof}

\subsection{Proof of Theorem \ref{T:fde}}

We are now ready to prove Theorem \ref{T:fde}.

\begin{proof}[Proof of Theorem \ref{T:fde}]
Let $\varphi \in C^\infty_c(\Rd)$ be such that $0 \leq \varphi \leq 1$ in $\R^\dim$, $\varphi \equiv 1$ in $\B_{2R}$ and $\varphi \equiv 0$ on $\Rd \setminus \B_{3R}$. We then note that
\begin{align*}
\int_{\B_{2R}} u_n(x,0)^{q-1} \, \d x
&= \int_{\B_{2R}} \mu_n(x) \, \d x\\
&\leq \int_{\Rd} \varphi(x) \mu_n(x) \, \d x
%&= \int_{\Rd} \varphi(x) \mu_n^+(x) \, \d x + \int_{\Rd} \varphi(x) \mu_n^-(x) \, \d x\\
 \to \int_{\Rd} \varphi(x) \, \d \mu(x) \leq \mu(\B_{3R})%+ \int_{\Rd} \varphi(x) \, \d \mu^-(x)
\end{align*}
for any $R > 0$, $t > 0$. Thus we have \eqref{MR} with $M_R = \mu(\B_{3R}) + 1 < +\infty$ for $n$ large enough. Hence by Lemma \ref{L:lL1}, there exists a constant $C > 0$ independent of $n$ such that
\begin{align}
\sup_{\tau \in (0,t)}\int_{\B_R} u_n(x,\tau)^{q-1} \, \d x
&\leq C \int_{\B_{2R}} \mu_n \, \d x + C \left( \frac{t}{R^{\kappa_1}} \right)^{\frac{q-1}{q-2}}\nonumber\\
%&\leq C \left( \int_{\B_{2R}} \mu_n^+ \varphi \, \d x + \int_{\B_{2R}} \mu_n^- \varphi \, \d x\right) + C \left( \frac{t}{R^{\kappa_1}} \right)^{\frac{q-1}{q-2}}\nonumber\\
&\to C \int_{\B_{2R}} \varphi \, \d \mu(x) + C \left( \frac{t}{R^{\kappa_1}} \right)^{\frac{q-1}{q-2}}\nonumber\\
&\leq C \mu(\B_{2R}) + C \left( \frac{t}{R^{\kappa_1}} \right)^{\frac{q-1}{q-2}}\label{fd:ini}
\end{align}
for any $R,t>0$. Hence by Corollary \ref{C:est}, (A1) and (A2) with \eqref{F-set} and $S = +\infty$ can be checked for energy solutions $u_n$ to \eqref{pde-n}--\eqref{ic-n}, $n \in \N$. Therefore, thanks to Theorem \ref{T:loc}, we can ensure the existence of a local-energy solution $u = u(x,t)$ to \eqref{pde}, \eqref{ic} on $(0,+\infty)$, and moreover, by a priori estimates for $(u_n)$ due to Lemma \ref{lLE}, the (weak star) limit $u$ turns out to be bounded in $\Rd \times [\vep,+\infty)$ for any $\vep > 0$. Furthermore, the quantitative estimate \eqref{fde:1} follows by a direct application of Lemma \ref{lLE} to the limit $u$, since every local-energy solution is a local weak solution for \eqref{pde}.  Recalling \eqref{fd:ini} and exploiting weak lower semicontinuity of norms, we can derive that
\begin{align*}
\sup_{\tau \in (\vep,t)} \int_{\B_R} u(x,\tau)^{q-1} \, \d x
&\leq \liminf_{n \to +\infty} \left(
\sup_{\tau \in (\vep,t)}\int_{\B_R} u_n(x,\tau)^{q-1} \, \d x \right) \nonumber\\
&\stackrel{\eqref{fd:ini}}\leq C \mu(\B_{2R}) + C \left( \frac{t}{R^{\kappa_1}} \right)^{\frac{q-1}{q-2}}
\end{align*}
for $\vep > 0$. Hence passing to the limit as $\vep\to0_+$ and using the fact that $u(\cdot,t)^{q-1} \in L^1(\B_R)$ for $t>0$, one obtains \eqref{fde:2}. Furthermore, recalling \eqref{e:H-L1} with $u = u_n$ and using \eqref{fd:ini} again, we find that
\begin{align}
\lefteqn{
\int^t_\vep \int_{\B_R} H(\nabla u) \, \d x \d \tau
\leq \liminf_{n \to +\infty} \int^t_\vep \int_{\B_R} H(\nabla u_n) \, \d x \d \tau
}\nonumber\\
&\leq C t^{\frac12} R^{\frac{N(q-2)}{2(q-1)}} \left[ \mu(\B_{2R}) +  \left( \frac{t}{R^{\kappa_1}} \right)^{\frac{q-1}{q-2}} \right]^{\frac{q}{2(q-1)}} + C t^{\frac{q-1}{q-2}} R^{\dim -\frac q {q-2}}.
\end{align}
%Here we also used the weak lower semicontinuity of the norm. 
Hence noting that $H(\nabla u) \in L^1(\B_R \times (0,t))$ and passing to the limit as $\vep \to 0_+$, we obtain \eqref{fde:3}.

As for the second half of the assertion, we take non-negative $v_{0,n} \in C^\infty_c(\Rd)$ such that $v_{0,n} \to v_0$ strongly in $L^p(\B_R)$ for any $R > 0$ and set $\mu_n = v_{0,n}$. By Lemma \ref{L:lLr} with $r=p$, for each $R , t \in (0,+\infty)$, one can obtain the following estimate:
$$
\sup_{\tau \in (0,t)} \int_{\B_R} u_n(x,\tau)^{p(q-1)} \, \d x
\leq C \int_{\B_{2R}} v_{0,n}(x)^p \, \d x + C \left(\frac{t^p}{R^{\kappa_p}}\right)^{\frac{q-1}{q-2}},
$$
provided that \eqref{hypo-p} is fulfilled, that is, $\kappa_p > 0$. Thus Lemma \ref{lLE} implies
$$
\sup_{x \in \B_R} u_n(x,t)^{q-1} \leq C t^{-\frac \dim {\kappa_p}} \left( 
\sup_{\frac t8 < \tau < t} \int_{\B_{2R}} u_n(x,\tau)^{p(q-1)} \, \d x 
\right)^{\frac 2 {\kappa_p}} 
+ C \left( \frac t {R^2}\right)^{\frac{q-1}{q-2}}
$$
for any $t > 0$ and $R > 0$. The rest of proof runs as in the first half of the proof.
\end{proof}

\section{Finsler PME}\label{S:FPME}

In this section, based on Theorem \ref{T:loc}, we shall prove Theorem \ref{T:pme} for the case
$$
1 < q < 2.
$$
Construction of approximate solutions $(u_n)$ can be performed as in \S \ref{Ss:aprx}. Hence (A0) follows. So we shall check (A1) and (A2) for $(u_n)$. As we will see later (see Lemma \ref{L:sc} below), it suffices to check them for locally bounded \emph{non-negative} local weak \emph{subsolutions} for general domains $\Omega \subset \Rd$ and $T > 0$, i.e., measurable functions $u = u(x,t)$ in $Q_T = \Omega \times (0,T)$ which comply with \eqref{es:1} and, instead of \eqref{es:3},
\begin{equation}\label{es:4}
- \iint_{Q_T} u^{q-1} \partial_t \varphi \, \d x \d t + \iint_{Q_T} H(\nabla u) \nabla_\xi H(\nabla u) \cdot \nabla \varphi \, \d x \d t \leq 0
\end{equation}
for all $\varphi \in C^\infty_c(Q_T)$ satisfying $\varphi \geq 0$ a.e.~in $Q_T$. It implies that, for every bounded domain $B \Subset \Omega$ and $T>0$,
\begin{equation}
\left\langle \partial_t u^{q-1}(t), w \right\rangle_{H^1_0(B)} + \int_B H(\nabla u(t)) \nabla_\xi H(\nabla u(t)) \cdot \nabla w \, \d x \leq 0
\label{es:4p}
\end{equation}
for a.e.~$t \in (0,T)$ and any $w \in H^1_0(B)$ satisfying $w \geq 0$ a.e.~in $B$. In the rest of this section, as an independent interest, we shall establish local estimates for such non-negative subsolutions.
%emph{non-negative measurable function} $u : \Omega \times (0,T) \to [0,+\infty)$ satisfying
%\begin{align}\label{es:1pm}
%u \in L^\infty_{\rm loc}(\Omega \times (0,T)) \cap L^2_{\rm loc}(0,T;H^1(B)),\quad
%%u^{q-1} &\in C((0,T);L^{q'}(B)) \cap C_{\rm weak}((0,T);L^2(B)),\label{es:2}
%u^{q-1} \in W^{1,2}_{\rm loc}(0,T;H^{-1}(B)),
%\end{align}
%where $H^{-1}(B)$ stands for the dual space of $H^1_0(B)$, (hence $u^{q-1} \in C_{\rm weak}(0,T;L^\infty(B))$) for any domain $B \Subset \Omega$ (i.e., $B$ is bounded and $\overline{B} \subset \Omega$) and
%\begin{equation}
%- \iint_{Q_T} u^{q-1} \partial_t \varphi \, \d x \d t + \iint_{Q_T} H(\nabla u) \nabla_\xi H(\nabla u) \cdot \nabla \varphi \, \d x \d t \leq 0,\label{es:3pm}
%\end{equation}
%where $Q_T := \Omega \times (0,T)$, for all $\varphi \in C^\infty_c(Q_T)$ satisfying $\varphi \geq 0$ in $Q_T$.

\begin{remark}[Sign-changing solutions for the Finsler fast diffusion equation]
{\rm
In contrast with the Finsler porous medium case, it does not seem available to apply the same strategy as in the present section to handle sign-changing weak solutions to the Finsler fast diffusion case, since the gradient estimate is established only for non-negative weak (super)solutions in Lemma \ref{L:lHL1}. As for classical fast diffusion (as well as porous medium) equations, it is not necessary to derive such a gradient estimate due to the linearity of the Laplace operator; indeed, one can employ the very weak formulation to derive an $L^1$-estimate (cf.~Lemma \ref{L:lL1}) and to identity weak limits, and hence, an existence result is obtained for possibly sign-changing data $\mu \in L^1_{\rm loc}(\Rd)$ (see~\cite{HerPi85}).
}
\end{remark}

\subsection{Local estimates for the Finsler PME}

%As in the last section, to check (A1) and (A2) with \eqref{F-set}, we shall establish local estimates for \emph{locally bounded non-negative local weak subsolutions} of the Finsler PME, that is, non-negative measurable functions $u=u(x,t)$ satisfying \eqref{es:1}, \eqref{es:4} for $1 < q < 2$ (see also Lemma \ref{L:sc} for extending estimates to sign-changing solutions). However, 
A significant difference between the FDE and PME cases will show up. More precisely, for the PME case, one cannot always obtain time-global estimates (cf.~Lemmas \ref{lLE}--\ref{L:lL1} for the FDE case) and one may need to squeeze a time-interval for local estimates depending on the growth of initial data at infinity.

To this end, for $r > 0$, we set
\begin{alignat*}{4}
\normtre{f}_r &:= \sup_{R \geq r} \left( R^{-\frac \kappa d} \int_{\B_R} |f(x)| \, \d x \right) \quad &&\mbox{ for } \ f \in L^1_{\rm loc}(\Rd),\\
\normtre{\mu}_r &:= \sup_{R \geq r} \left( R^{-\frac \kappa d} |\mu|(\B_R) \right) &&\mbox{ for } \ \mu \in \mathcal M(\Rd),
\end{alignat*}
where
$$
\kappa := \kappa_1 = 2+Nd \quad \mbox{ and } \quad d := \frac{2-q}{q-1} > 0.
$$
Moreover, when $f$ is not defined on the whole of $\Rd$, e.g., $f \in L^1_{\rm loc}(\Omega)$, the integrand $f(x)$ in the definition of $\normtre{f}_r$ above will be replaced with its zero extension onto $\Rd$. Then a key lemma reads,

\begin{lemma}[Local estimates for the Finsler PME]\label{pm:lLE}
Let $u = u(x,t) : Q_T \to [0,+\infty)$ be a non-negative measurable function satisfying \eqref{es:1}, \eqref{es:4} with $1 < q < 2$ and assume for any $r > 0$ that
\begin{equation}\label{pm-ini}
\liminf_{s \to 0_+}\left( R^{-\frac \kappa d} \int_{\B_R\cap\Omega} u(x,s)^{q-1} \, \d x \right) \lesssim \normtre{\mu}_r
\end{equation}
for all $R > r$ and
\begin{equation}\label{phi_psi_wd}
 \sup_{\tau \in (0,t)} \normtre{u(\cdot,\tau)^{q-1}}_r < +\infty, \quad 
 \sup_{\tau \in (0,t)} \sup_{R \geq r} \left( \tau^{\frac N \kappa} R^{-\frac 2d} \|u(\cdot,\tau)\|_{L^\infty(\B_R\cap\Omega)}^{q-1}\right) < +\infty,
\end{equation}
for all $t \in (0,T)$. Then, for any $r > 0$, there exist a time $\Tr(\mu) \simeq \normtre{\mu}_r^{-d} > 0$ and a constant $C \geq 0$ such that 
\begin{enumerate}
 \item an estimate for $\normtre{\cdot}_r$\/{\rm :}
$$
\displaystyle \normtre{u(t)^{q-1}}_r \leq C \normtre{\mu}_r,
$$
 \item an $L^\infty$ estimate\/{\rm :}
$$
\|u(t)\|_{L^\infty(\B_R)}^{q-1} \leq C t^{-\frac N \kappa} R^{\frac 2 d} \normtre{\mu}_r^{\frac 2 \kappa},
$$
 \item an $L^1$ estimate for gradients\/{\rm :}
$$
\int^t_0 \|H(\nabla u)\|_{L^1(\B_R)} \, \d \tau
\leq C t^{\frac 1 \kappa} R^{1+\frac \kappa d} \normtre{\mu}_r^{1+\frac d \kappa}
$$
\end{enumerate}
hold true for any $t \in (0, \Tr(\mu) \wedge T)$ and $R \in [r,+\infty)$ satisfying $\B_{2R} \subset \Omega$.
\end{lemma}

To prove this lemma, we first set up the following
\begin{lemma}\label{pm:L:1}
Let $u = u(x,t) : Q_T \to [0,+\infty)$ be a non-negative measurable function satisfying \eqref{es:1}, \eqref{es:4} with $1 < q < 2$. Then there exists a constant $C > 0$ such that 
\begin{align*}
 \|u(t)\|_{L^\infty(\B_R)} \leq C \left( \frac{\|u\|_{L^\infty(\B_{2R} \times (\frac t 8,t))}^{2-q}}{R^2} + \frac 1 t \right)^{\frac{N+2}{N(2-q)+4}} \left( \int^t_{t/8} \int_{\B_{2R}} u^2 \, \d x \d \tau \right)^{\frac 2{N(2-q)+4}}
\end{align*}
for any $R, t > 0$ satisfying $\B_{2R} \subset \Omega$.
\end{lemma}

\begin{proof}[Outline of proof]
The lemma can be proved based on the argument of~\cite[Proof of Lemma 3.1, p.198]{DiHe89} (see also~\cite[Lemma 5.1, p.1253]{Ishige96}). To be more precise, as in the proof of Lemma \ref{lLE}, testing \eqref{es:4p} with $w=(u-k_n)_+ \zeta_n^2$ and repeating a similar argument, one can prove an estimate for $\|u\|_{L^\infty(Q_\infty)}$ with a bound which still involves $\|u\|_{L^\infty(Q_0)}$ with a power less than $1$ as well as $\|u\|_{L^2(Q_0)}$ for some subdomains $Q_\infty \subsetneq Q_0 \subsetneq Q_T$. On the other hand, such a dependence of the bound on $\|u\|_{L^\infty(Q_0)}$ differs from that in the proof of Lemma \ref{lLE}, and in particular, the $L^\infty$ norm will remain in the bound even after performing iteration arguments.
\end{proof}

Now, we shall introduce the following two important quantities, which will enable us to measure the growth of solutions at (spatial) infinity in an asymptotic way:
\begin{align*}
 \psir(t) &:= \sup_{\tau \in (0,t)} \normtre{u(\tau)^{q-1}}_r
= \sup_{\tau \in (0,t)} \sup_{R \geq r} \left( R^{-\frac \kappa d} \int_{\B_R\cap\Omega} u(x,\tau)^{q-1} \, \d x \right),\\
 \phir(t) &:= \sup_{\tau \in (0,t)} \sup_{R \geq r} \left( \tau^{\frac N \kappa} R^{-\frac2d} \|u(\tau)\|_{L^\infty(\B_R\cap\Omega)}^{q-1} \right),
\end{align*}
which are non-decreasing in $t$. In what follows, we always assume \eqref{phi_psi_wd}. Hence $\psir(t)$ and $\phir(t)$ are finite for any $t \in (0,T)$ and $r > 0$ (here we remark that the approximate solutions $(u_n)$ enjoy the assumption above with $\Omega = \B_n$).

The next lemma can be proved in an analogous fashion to the proof of Lemma \ref{L:lHL1}. More precisely, we refer the reader to~\cite[Lemma 3.3, p.200]{DiHe89}.

\begin{lemma}[$L^1$ estimate for gradients]\label{pm:L:grad-L1}
Let $u = u(x,t) : Q_T \to [0,+\infty)$ be a non-negative measurable function satisfying \eqref{es:1}, \eqref{es:4} and \eqref{phi_psi_wd} with $1 < q < 2$. Let $\zeta : \Omega \to [0,1]$ be a smooth cut-off function such that ${\mathrm supp}\, \zeta = \overline{\B_{2R}}$ and $\zeta \equiv 1$ on $\B_R$. Then it holds that
\begin{align*}
\lefteqn{
 \int^t_0 \int_{\B_{2R}} H(\nabla u) \zeta \, \d x \d \tau
}\\
&\leq C_{q,N} \left(
R^{1+\frac \kappa d} \int^t_0 \tau^{\frac 1 \kappa - 1} \phir(\tau)^{\frac d 2} \psir(\tau) \, \d \tau + R^{1+\frac \kappa d} \int^t_0 \tau^{\frac 1 2 -  \frac{3Nd}{2\kappa}} \phir(\tau)^{\frac 3 2 d} \psir(\tau) \, \d \tau\right)^{1/2}\\
&\quad \times \left( R^{1+\frac\kappa d} \int^t_0 \tau^{\frac 1 \kappa - 1} \phir(\tau)^{\frac d 2} \psir(\tau) \, \d \tau \right)^{1/2}
\end{align*}
for all $t \in (0,T)$ and $r,R \in \R$ satisfying $2R \geq r > 0$ and $\B_{2R} \subset \Omega$.
\end{lemma}

We shall next derive integral inequalities for $\psir(t)$ and $\phir(t)$. They will finally play a crucial role to prove Lemma \ref{pm:lLE}.

\begin{lemma}[Integral inequality for $\phir$]\label{pm:e-1} Under the same setting as above, there exist constants $C_1,C_2 > 0$ such that
\begin{align}\label{phi-ineq}
 \phir(t) \leq C_1 \int^t_0 \tau^{-\frac{Nd}\kappa} 
\phir(\tau)^{\frac 1 {q-1}} \, \d \tau + C_2 \psir(t)^{\frac 2 \kappa}
\end{align}
for any $t > 0$ and $r > 0$.
\end{lemma}

\begin{proof}[Outline of proof]
The lemma can be proved as in~\cite[Lemma 3.2., p.200]{DiHe89}. Set
$$
\lam := Nd + \frac 4 {q-1}.
$$
By Lemma \ref{pm:L:1}, noting that
$$
R^{-2} \|u(t)\|_{L^\infty(\B_{2R} \times (t/8,t))}^{2-q} + t^{-1}
\leq t^{-\frac{Nd}\kappa} \phir(t)^d + t^{-1},
$$
we have
\begin{align*}
 t^{\frac N \kappa} R^{-\frac 2 d} \|u(t)\|_{L^\infty(\B_R)}^{q-1}
 &\lesssim t^{\frac N \kappa} R^{-\frac 2 d} \left(t^{-\frac{Nd}\kappa} \phir(t)^d + t^{-1}\right)^{\frac{N+2}\lambda} \left( \int^t_{t/8} \int_{\B_{2R}} u^2 \, \d x \d t \right)^{\frac 2 \lambda}\\
 &\lesssim H_1 + H_2,
\end{align*}
where $H_1$ and $H_2$ are given by
\begin{align*}
 H_1 &:= t^{\frac N \kappa} R^{-\frac 2 d} t^{-\frac{Nd}\kappa\cdot\frac{N+2}\lam} \phir(t)^{\frac{d(N+2)}\lam} \bigg( \int^t_{t/8} \underbrace{\int_{\B_{2R}} u(x,\tau)^2 \, \d x}_{\leq \ (2R)^N \|u(\tau)\|_{L^\infty(\B_{2R})}^2} \d \tau \bigg)^{\frac 2 \lam},\\
 H_2 &:= t^{\frac N \kappa} R^{-\frac 2 d} t^{-\frac{N+2}\lam} \bigg( \int^t_{t/8} \underbrace{\int_{\B_{2R}} u(x,\tau)^{q-1} u(x,\tau)^{3-q} \, \d x}_{\leq \ \|u(\tau)\|_{L^\infty(\B_{2R})}^{3-q} \int_{\B_{2R}} u(x,\tau)^{q-1} \, \d x} \d \tau \bigg)^{\frac 2 \lam}.
\end{align*}
Then it suffices to estimate $H_1$ and $H_2$ from above in terms of $\phir(t)$ and $\psir(t)$ by the use of the definitions of these quantities only. Finally, we obtain the assertion by taking a supremum of both sides in $R \geq r > 0$ and $t \in (0,T)$. 
\end{proof}

We next derive the second integral inequality.

\begin{lemma}[Integral inequality for $\psir$]\label{pm:e-2}
Under the same setting as above, there exist constants $C_3, C_4 > 0$ such that
\begin{align}\label{psi-ineq}
 \psir(t) &\leq C_3 \normtre{\mu}_r \nonumber \\
 &\quad + C_4 \left( \int^t_0 \tau^{\frac 1 \kappa - 1} \phir(\tau)^{\frac d 2} \psir(\tau) \, \d \tau + \int^t_0 \tau^{\frac 3 \kappa - 1} \phir(\tau)^{\frac 3 2 d} \psir(\tau) \, \d \tau \right)
\end{align}
for any $t > 0$ and $r > 0$.
\end{lemma}

\begin{proof}[Outline of proof]
The lemma can be verified as in the proof of~\cite[Lemma 3.4., p.202]{DiHe89}. So we just give an outline. Let $\zeta_R$ be a smooth cut-off (in space only) function supported over the set $\overline{\B_{2R}}$ and substitute $w = \zeta_R^2$ in \eqref{es:4p}. Then we have
\begin{align*}
 \frac{\d}{\d t} \left( \int_{\B_{2R}} u^{q-1} \zeta_R^2 \, \d x\right)
+ 2 \int_{\B_{2R}} (H \nabla_\xi H)(\nabla u) \cdot (\nabla \zeta_R) \zeta_R \, \d x \leq 0.
\end{align*}
To handle the second term in the left-hand side, we recall Lemma \ref{pm:L:grad-L1}. Accordingly, it enables us to estimate
$$
R^{-\frac \kappa d} \int_{\B_R} u(x,t)^{q-1} \, \d x
$$
(cf.~the definition of $\psi_r(t)$) from above by use of the definitions of $\phi_r(t)$ and $\psi_r(t)$.
\end{proof}

We are now in a position to prove Lemma \ref{pm:lLE}. To exhibit how $\Tr(\mu)$ is selected, we shall give a complete proof for the lemma, although it is similar to~\cite[Lemma 3.5, p.202]{DiHe89}.

\begin{proof}[Proof of Lemma \ref{pm:lLE}]
Fix $t_* > 0$ arbitrarily. Since $t \mapsto \psir(t)$ is non-decreasing by definition, we derive from \eqref{phi-ineq} that
\begin{align*}
 \phir(t) \leq C_1\int^t_0 \tau^{-\frac{Nd}{\kappa}} \phir(\tau)^{\frac 1 {q-1}} \, \d \tau + C_2 \psir(t_*)^{\frac 2 \kappa}
\end{align*}
for $t \in (0,t_*)$. Then $\phir$ can be majorized by the solution of the Cauchy problem,
$$
H_\vep'(t) = C_1 t^{-\frac{Nd}{\kappa}} H_\vep(t)^{\frac 1 {q-1}}, \quad H_\vep(0) = C_2 \psir(t_*)^{\frac 2 \kappa}+\vep,
$$
whose solution is given by
$$
H_\vep(t) = \left[ \left(C_2 \psir(t_*)^{\frac 2 \kappa}+\vep\right)^{-d} - \frac{d\kappa}2 C_1t^{\frac{2}\kappa}\right]_+^{-\frac 1 d}
$$
(see also Lemma \ref{A:L:cp} in Appendix). Hence letting $\vep \to 0_+$, we infer that
$$
\phir(t) \leq \left( C_2^{-d} \psir(t_*)^{- \frac{2d}{\kappa}} - \frac{d\kappa}2 C_1 t^{\frac2\kappa}\right)_+^{-\frac 1 d}
\quad \mbox{ for } \ t \in (0,t_*),
$$
which provides a bound of $\phir(t)$, unless the quantity in the parentheses is non-positive (then no information comes out any longer). In particular, setting $t_* = t$, we have
$$
\phir(t) \leq \psir(t)^{\frac 2 \kappa }\left( C_2^{-d} - \frac{d\kappa}2 C_1 t^{\frac 2 \kappa} \psir(t)^{\frac{2d}\kappa} \right)_+^{-\frac 1 d},
$$
and hence, for any $t > 0$ satisfying $( C_2^{-d} - \frac{d\kappa}2 C_1 t^{\frac 2 \kappa} \psir(t)^{\frac{2d}\kappa} )_+^{-\frac 1 d} \leq 2C_2$, that is,
\begin{equation}\label{psi-cond}
\left[t\psir(t)^d\right]^{\frac{2}{\kappa}} \leq \frac{2}{d\kappa C_1} \left(1 - 2^{-d}\right) C_2^{-d},
\end{equation}
we have
\begin{equation}\label{mo}
\phir(t) \leq 2 C_2 \psir(t)^{\frac 2 \kappa}. 
\end{equation}
Therefore it follows that
\begin{align*}
 \psir(t) &\stackrel{\eqref{psi-ineq}}\leq C_3 \normtre{\mu}_r + C_4\left( \int^t_0 \tau^{\frac 3 \kappa - 1} \phir(\tau)^{\frac 3 2 d} \psir(\tau) \, \d \tau + \int^t_0 \tau^{\frac 1 \kappa - 1} \phir(\tau)^{\frac d 2} \psir(\tau) \, \d \tau\right)\\
&\stackrel{\eqref{mo}}\leq C_3 \normtre{\mu}_r + C_4 \left[ (2C_2)^{\frac32d} \int^t_0 \tau^{\frac 3 \kappa - 1} \psir(\tau)^{1 + \frac{3d} \kappa} \, \d \tau + (2C_2)^{\frac d2} \int^t_0 \tau^{\frac 1 \kappa - 1} \psir(\tau)^{1+\frac d \kappa} \, \d \tau\right].
\end{align*}
Here we note that
$$
\tau^{\frac 3 \kappa - 1} \psir(\tau)^{1 + \frac{3d} \kappa}
= \tau^{\frac 1 \kappa - 1} \psir(\tau)^{1+\frac d \kappa} \left[ \tau \psir(\tau)^d \right]^{\frac 2 \kappa},
$$
which along with \eqref{psi-cond} gives
$$
\psir(t) \leq C_3 \normtre{\mu}_r + C_5 \int^t_0 \tau^{\frac 1 \kappa - 1} \psir(\tau)^{1+\frac d \kappa} \, \d \tau
$$
for some constant $C_5 > 0$. Then $\psir(t)$ is majorized by the solution
$$
G_\vep(t) = \left[ (C_3\normtre{\mu}_r + \vep)^{-\frac d\kappa} - C_5 d t^{\frac 1 \kappa}\right]^{-\frac \kappa d}_+
$$
of the Cauchy problem,
$$
G'(t) = C_5 t^{\frac 1 \kappa - 1} G^{1+\frac d \kappa}, \quad G_\vep(0) = C_3 \normtre{\mu}_r + \vep
$$
(see also Lemma \ref{A:L:cp} in Appendix). Thus passing to the limit as $\vep \to 0_+$, one obtains
$$
\psir(t) \leq \left[( C_3 \normtre{\mu}_r )^{-\frac d\kappa} - C_5 d t^{\frac 1 \kappa}\right]^{-\frac \kappa d}_+
= \normtre{\mu}_r \left( C_3^{-\frac d \kappa} - C_5 d \normtre{\mu}_r^{\frac d\kappa} t^{\frac 1 \kappa} \right)^{-\frac \kappa d}_+,
$$
which implies
\begin{equation}\label{III}
\normtre{u(t)^{q-1}}_r \leq \psir(t) \leq C \normtre{\mu}_r 
\end{equation}
whenever 
\begin{equation}\label{pm:ini}
C_5 d \normtre{\mu}_r^{\frac d \kappa} t^{\frac 1 \kappa} \leq \frac12 C_3^{-\frac d \kappa}, \ \mbox{ i.e., } \
0 \leq t \leq \left( \frac{C_3^{-\frac d \kappa}}{2 C_5 d \normtre{\mu}_r^{\frac d \kappa}}\right)^\kappa \simeq \normtre{\mu}_r^{-d}.
\end{equation}
Thus (i) follows. In particular, the estimate above yields
\begin{equation}\label{UUU}
t \psir(t)^d \lesssim t \normtre{\mu}_r^{d},
\end{equation}
and hence, one can take a constant $\Tr(\mu) \simeq \normtre{\mu}_r^{-d} > 0$ as in \eqref{Tmu} such that \eqref{psi-cond} as well as \eqref{pm:ini} hold for any $t \in (0,\Tr(\mu))$. Moreover, it follows from \eqref{mo} that
$$
\|u(t)\|_{L^\infty(\B_R)}^{q-1} \lesssim t^{-\frac N \kappa} R^{\frac 2 d} \normtre{\mu}_r^{\frac 2 \kappa} \quad \mbox{ for any } \ t \in (0,\Tr(\mu)),
$$
which implies (ii). Finally, by Lemma \ref{pm:L:grad-L1} along with \eqref{psi-cond}--\eqref{III} and \eqref{UUU}, one can verify that
\begin{align}
 \lefteqn{
 \int^t_0 \|H(\nabla u)\|_{L^1(\B_R)} \, \d \tau
}\nonumber\\
 &\leq C R^{1+\frac \kappa d} \left(
\int^t_0 \tau^{\frac 1 2 - \frac{3Nd}{2\kappa}} \phir(\tau)^{\frac 3 2 d} \psir(\tau) \, \d \tau  + \int^t_0 \tau^{\frac 1 \kappa - 1} \phir(\tau)^\frac d 2 \psir(\tau) \, \d \tau \right)\nonumber\\
&\leq C R^{1+\frac \kappa d} \left(\int^t_0 \tau^{\frac 1 \kappa - 1} \, \d \tau\right) \normtre{\mu}_r^{1+\frac d \kappa}\nonumber \quad \mbox{(by \eqref{psi-cond}--\eqref{III} and \eqref{UUU})}\\
&\leq C R^{1+\frac \kappa d} t^{\frac 1 \kappa} \normtre{\mu}_r^{1+\frac d \kappa}
\quad \mbox{ for any } \ t \in (0,\Tr(\mu)).\label{A}
\end{align}
Here we used the fact that
\begin{align*}
 \tau^{\frac12 - \frac{3Nd}{2\kappa}} \phi_r(\tau)^{\frac32d} \psi_r(\tau)
&\stackrel{\eqref{mo}}\lesssim \tau^{\frac3\kappa-1}\psi_r(\tau)^{\frac3\kappa d+1}\\
&= \left[\tau \psi_r(\tau)^d\right]^{\frac2\kappa}\tau^{\frac1\kappa-1} \psi_r(\tau)^{1+\frac d\kappa}
\stackrel{\eqref{psi-cond}}\lesssim \tau^{\frac1\kappa-1} \psi_r(\tau)^{1+\frac d\kappa}.
\end{align*}
This completes the proof.
\end{proof}

We further obtain the following corollary:

\begin{corollary}[Local estimates for gradients]\label{pm:lGE}
Let $u = u(x,t) : Q_T \to [0,+\infty)$ be a non-negative measurable function satisfying \eqref{es:1}, \eqref{es:4}, \eqref{pm-ini} and \eqref{phi_psi_wd} with $1 < q < 2$. Then for each $r>0$ it holds that
\begin{align}
\int^t_0 \tau^{\frac12} \left( \int_{\B_R} H(\nabla u^{\frac{4-q}4})^2 \zeta_R^2\, \d x\right) \, \d \tau
\lesssim R^{1+\frac\kappa d} t^{\frac 1 \kappa} \normtre{\mu}_r^{1+\frac d\kappa}\label{pm:HuL2}
\end{align}
for any $t \in (0,\Tr(\mu))$ and $R \in [r,+\infty)$ satisfying $\B_{2R} \subset \Omega$. Furthermore, for any $0 < t_1 < t_2 < \Tr(\mu)$ and $R \in [r,+\infty)$ satisfying $\B_{2R} \subset \Omega$, there exists a constant $M \geq 0$ depending only on $q,N,\normtre{\mu}_r,t_1,t_2,r,R$ such that
\begin{equation}\label{pm:HL2}
\int^{t_2}_{t_1} \int_{\B_R} H(\nabla u)^2 \, \d x \d \tau \leq M.
\end{equation}
\end{corollary}

\begin{proof}[Outline of proof]
Indeed, as a by-product of the proof for Lemma \ref{pm:L:grad-L1}, one can derive
\begin{align*}
\lefteqn{
\int^t_0 \tau^{\frac12} \left( \int_{\B_{2R}} H(\nabla u^{\frac{4-q}4})^2 \zeta_R^2\, \d x\right) \, \d \tau
}\\
&\leq R^{1+\frac \kappa d} \int^t_0 \tau^{\frac 1 2 - \frac{3 N d}{2\kappa}} \phir(\tau)^{\frac 3 2 d} \psir(\tau) \, \d \tau
 + R^{1 + \frac \kappa d} \int^t_0 \tau^{\frac 1 \kappa - 1} \phir(\tau)^{\frac d 2} \psir(\tau) \, \d \tau.
\end{align*}
Hence using \eqref{mo}--\eqref{III} and \eqref{UUU}, we can obtain \eqref{pm:HuL2}. Moreover, \eqref{pm:HL2} follows from \eqref{pm:HuL2} and (ii) of Lemma \ref{pm:lLE}. 
\end{proof}

Now, all the estimates obtained so far can be extended to (possibly \emph{sign-changing}) locally bounded local weak \emph{solutions} to \eqref{es:1} and \eqref{es:3} for $1 < q < 2$. 
%$$
%- \iint_{Q_T} u^{q-1} \partial_t \varphi \, \d x \d t + \iint_{Q_T} H(\nabla u) \nabla_\xi H(\nabla u) \cdot \nabla \varphi \, \d x \d t = 0
%$$
%for all $\varphi \in C^\infty_c(Q_T)$, in the following lemma. %It is also valid for the porous medium case, $1 < q < 2$.

\begin{lemma}[From ``\emph{non-negative subsolution}'' to ``\emph{sign-changing solution}'']\label{L:sc}
All the assertions of Lemma \ref{pm:lLE} and Corollary \ref{pm:lGE} also hold for {\rm (}possibly \emph{sign-changing}{\rm )} locally bounded local weak \emph{solutions} to \eqref{es:1} and \eqref{es:3} for $1 < q < 2$. 
\end{lemma}

\begin{proof}
All the estimates obtained in Lemma \ref{pm:lLE} and Corollary \ref{pm:lGE} have been proved for non-negative measurable functions satisfying \eqref{es:1}, \eqref{es:4}. Now, let $u$ be a locally bounded local weak solution and define the {\em positive-part} $u_+$ and {\em negative-part} $u_-$ of $u$ by
$$
u_+ := u \vee 0 \geq 0, \quad u_- := (-u) \vee 0 \geq 0
$$
(hence $u = u_+ - u_-$ and $|u| = u_+ + u_-$), which turn out to be non-negative measurable functions satisfying \eqref{es:1}, \eqref{es:4}. Moreover, note that
$$
\int_B |u|^r \, \d x = \int_B u_+^r \, \d x + \int_B u_-^r \, \d x
$$
for any $r \in [1,+\infty)$ and domain $B$ in $\R^\dim$. Furthermore, we also observe that $u_\pm \in W^{1,r}(B)$ for $u \in W^{1,r}(B)$ with $r \in [1,+\infty)$, and moreover,
$$
|\nabla u| = |\nabla u_+ - \nabla u_-| \leq |\nabla u_+| + |\nabla u_-|.
$$
Thus all the assertions of Lemma \ref{pm:lLE} and Corollary \ref{pm:lGE} have been extended to {\rm (}possibly sign-changing{\rm )} locally bounded local weak solutions. 
\end{proof}

\subsection{Proof of Theorem \ref{T:pme}}

We are in a position to prove Theorem \ref{T:pme}.

\begin{proof}[Proof of Theorem \ref{T:pme}]
Let $u_n$ be (the zero extension onto $\Rd$ of) an energy solution of \eqref{pde-n}--\eqref{ic-n} on $[0,+\infty)$ (see \S \ref{Ss:aprx}). Fix $r > 0$. For $n \in \mathbb{N}$ large enough, we find that \eqref{pm-ini} and \eqref{phi_psi_wd} hold true. Indeed, letting $\varphi \in C^\infty_c(\Rd)$ such that $0 \leq \varphi \leq 1$ in $\Rd$, $\varphi \equiv 1$ in $\B_R$ and $\varphi \equiv 0$ on $\Rd \setminus \B_{2R}$, we see that
\begin{align*}
\lefteqn{
 \liminf_{t \to 0_+} \int_{\B_R} |u_n(x,t)|^{q-1} \, \d x
}\\
&\leq \liminf_{t \to 0_+} \int_{\Rd} |u_n(x,t)|^{q-1} \varphi(x)\, \d x\\
&= \int_{\Rd} |\mu_n(x)| \varphi(x) \, \d x
\to \int_{\Rd} \varphi(x) \, \d |\mu|(x) \leq \mu(\B_{2R}).
\end{align*}
It follows that
$$
\liminf_{t \to 0_+} R^{-\frac\kappa d} \int_{\B_R} |u_n(x,t)|^{q-1} \, \d x 
\lesssim (2R)^{-\frac\kappa d} |\mu|(\B_{2R})
\leq \normtre{\mu}_r
$$
for $2R \geq r$ and $n \in \mathbb{N}$ large enough. Thus it yields \eqref{pm-ini} for $u = u_n$ with $n \gg 1$. Moreover, \eqref{phi_psi_wd} readily follows from the regularity for $u_n$ mentioned in \S \ref{Ss:aprx}.

Then, thanks to (i) and (iii) of Lemma \ref{pm:lLE}, we observe that
\begin{align*}
 \int^t_0\int_{\B_R} |u_n(x,\tau)^{q-1}| \, \d x \d \tau &\lesssim  t R^{\frac \kappa d} \normtre{\mu}_r,\\
 \int^t_0 \|H(\nabla u_n)\|_{L^1(\B_R)} \, \d \tau &\lesssim t^{\frac 1 \kappa} R^{1+\frac \kappa d} \normtre{\mu}_r^{1+\frac d \kappa},
\end{align*}
for $R > n$ and $t \in (0,\Tr(\mu))$. Thus (A2) with \eqref{F-set} has been checked for $S = \Tr(\mu)$. Furthermore, by (ii) of Lemma \ref{pm:lLE}, it follows that
\begin{align*}
\sup_{t \in (t_1,t_2)} \left( \int_{\B_R} |u_n(x,t)| \, \d x\right)
&\lesssim t_1^{-\frac N {\kappa(q-1)}} R^{\frac 2 {d(q-1)} + N} \normtre{\mu}_r^{\frac 2{\kappa(q-1)}},\\
\int^{t_2}_{t_1} \int_{\B_R} |u_n(x,t)^{q-1}|^2 \, \d x \d t
&\lesssim t_1^{-\frac{N}{\kappa}}(t_2-t_1) R^{\frac 4 d + N} \normtre{\mu}_r^{\frac{4}{\kappa}},
\end{align*}
for any $0 < t_1 \leq t_2 < \Tr(\mu)$. We recall \eqref{pm:HL2} as well. Therefore (A1) follows with \eqref{F-set} and $S = \Tr(\mu)$. Consequently, thanks to Theorem \ref{T:loc}, we can assure that \eqref{pde}, \eqref{ic} with $q \in (1,2)$ admits a local-energy solution on $(0,\Tr(\mu))$. Furthermore, passing to the limits of all the estimates in Lemma \ref{pm:lLE} with $u$ replaced by $u_n$ as $n \to +\infty$ and using weak lower semicontinuity of norms, we can prove \eqref{pme:1}--\eqref{pme:3} as in the proof of Theorem \ref{T:fde}. Finally, due to \eqref{Tmu}, one can also verify that \eqref{pde}, \eqref{ic} admits a local-energy solution on $(0,\T(\mu))$ in the sense of Definition \ref{D:sol}, since for any $T \in (0,\T(\mu))$ one can take $r > 0$ such that $uT < \Tr(\mu)$.
\end{proof}

\section{Further applications and extensions}\label{S:other}

The last section is devoted to other possible applications 
as well as possible extensions of the general framework developed in the present paper.

\subsection{Finsler heat equation}\label{Ss:fhe}

Theorem \ref{T:loc} is of course applicable to the Cauchy problem for the \emph{Finsler heat equation}, that is, \eqref{pde} with $q=2$, which has already been studied by the authors (see also~\cite{OS1}). In~\cite{AIS}, the authors studied \begin{equation}\label{F-heat}
 \partial_t u = \Delta_H u \ \mbox{ in } \Rd \times (0,+\infty), \quad u|_{t=0} = \mu,
\end{equation}
and proved existence of a distributional solution $u=u(x,t)$ to the Cauchy problem \eqref{F-heat} in $(0,1/(4\Lambda))$ for any signed Radon measure $\mu$ in $\Rd$ satisfying a square-exponential condition in terms of the dual norm $H_0(x)$,
\begin{equation}\label{gr}
  \sup_{x\in \Rd}\int_{B_{H_0}(x,1/\sqrt{\Lambda})} \e^{-\Lambda H_0(y)^2}\,\d|\mu|(y)<+\infty 
\end{equation}
for some $\Lambda>0$. 
For the validity of the assumption~\eqref{gr}, see \cite[Theorem~1.2-(i)]{AIS}. 
(See also e.g., \cite[Section~9]{A} and \cite[Theorem~1.8]{IK}.)
Moreover, it is also proved that $1/(4\Lambda)$ is the optimal maximal existence time and the growth condition \eqref{gr} is also optimal in a proper sense. To this end, for approximate solutions $(u_n)$, the following local estimates are established for some $T_* > 0$: there exist constants $\sigma, C > 0$ and $\ell \in (0,1/2)$ such that
\begin{align*}
\sup_{x \in \Rd} \int_{x+\B_1} \e^{-h(y,t)} |u_n(y,t)| \, \d y &\leq C\sup_{x \in \Rd} \int_{x+\B_1} \e^{-\Lambda H_0(y)^2} |\mu_n(y)| \, \d y,\\
\sup_{x \in \Rd} \int^t_0 \int_{x+\B_1} \e^{-h(y,\tau)} H(\nabla u_n(y,\tau)) \, \d y \d \tau &\leq C t^\sigma \sup_{x \in \Rd} \int_{x+\B_1} \e^{-\Lambda H_0(y)^2} |\mu_n(y)| \, \d y,
\end{align*}
where $h(y,\tau) := \Lambda H_0(y)^2(1+\tau^\ell)$, for any $t \in (0,T_*)$, and moreover, other local estimates for $(u_n)$ (e.g., estimates in $L^\infty_{\rm loc}(0,\Sl;L^2(\B_R)) \cap L^2_{\rm loc}(0,\Sl;H^1(\B_R))$ for $R > 0$) are also derived (hence (A1) and (A2) follow immediately). Finally, a $C^{1,\alpha}$-regularity estimate for quasilinear parabolic equations (with uniform elliptic operators) is employed to derive pointwise convergence of gradients $(\nabla u_n)$, and hence, it enables us to identify the limit of the gradient nonlinearity due to the coincidence between weak and pointwise limits. In contrast, the approach developed in the present paper does not require any higher regularity estimates (such as $C^{1,\alpha}$ estimates) and the identification of weak limits of the gradient nonlinearity can be performed within the framework of local estimates established in~\cite{AIS} only.
 
\subsection{Doubly-nonlinear parabolic equations}\label{Ss:dnp}

Theorem \ref{T:loc} is also applicable to other doubly-nonlinear parabolic equations, whose typical example reads,
\begin{equation}
 \partial_t \left( |u|^{q-2}u \right) = \Delta_p u \quad \mbox{ in } \Rd \times (0,+\infty),\label{dnp-pde}
\end{equation}
where $1 < q < +\infty$ and $\Delta_p$ is the so-called \emph{$p$-Laplacian} given by
$$
\Delta_p u = \mathrm{div} \left( |\nabla u|^{p-2} \nabla u\right), \quad 1 < p < +\infty.
$$
Equation \eqref{dnp-pde} can be regarded as a unified form of the PME/FDE \eqref{pmfd-2} as well as the $p$-Laplace parabolic equation,
$$
\partial_t u = \Delta_p u \ \mbox{ in } \Rd \times (0,+\infty).
$$
The Cauchy problem for \eqref{dnp-pde} with growing initial data has been studied in~\cite{Ishige96}, where existence of solutions is proved for the three cases $p>q$, $p=q$ and $p<q$.

The reduction of \eqref{dnp-pde} to \eqref{dnp} can be performed by setting
$$
\beta(u) = |u|^{q-2}u \quad \mbox{ and } \quad a(x,t,\xi) = |\xi|^{p-2}\xi.
$$
Here a potential of $\beta$ is given by
$$
\hat \beta(u) = \frac 1 q |u|^q.
$$
Hence \eqref{a-mono}--\eqref{beta-conv} can be checked immediately as in \S \ref{S:FFDE}. Moreover, (A0) can be checked with the aid of a general theory developed in~\cite{B-dn}, which enables us to guarantee existence of energy solutions of the Cauchy-Dirichlet problem for \eqref{dnp-pde} posed on any smooth bounded domain. Therefore existence of local-energy solutions can be ensured by Theorem \ref{T:loc}, if all the assumptions (A1) and (A2) are checked (see~\cite{Ishige96} for local-energy estimates). It is noteworthy that, thanks to Theorem \ref{T:loc}, it suffices to establish at most a local $L^2$-estimate for gradients $(\nabla u_n)$.

A further application may be extended to the \emph{Finsler doubly-nonlinear problem}, that is,
$$
 \partial_t \left( |u|^{q-2}u \right) = \Delta_{H,p} u \quad \mbox{ in } \Rd \times (0,+\infty),%\label{F-dnp}
$$
where $\Delta_{H,p}$ is the \emph{Finsler $p$-Laplacian} given by
$$
\Delta_{H,p} u := \mathrm{div} \left( H(\nabla u)^{p-1}\nabla_\xi H(\nabla u) \right), \quad 1 < p < +\infty.
$$
Then we set $a(x,t,\xi) = H(\xi)^{p-1}\nabla_\xi H(\xi)$. Therefore \eqref{a-mono}--\eqref{beta-conv} can also be checked, and moreover, (A0) can be proved by use of the abstract theory developed in~\cite{B-dn}. Hence it remains to prove (A1) and (A2) to apply Theorem \ref{T:loc}. On the other hand, proofs of all lemmas and corollaries in \S \ref{S:FFDE} and \S \ref{S:FPME} are free from the restriction $p = 2$, because the Finsler FDE and PME are already a sort of doubly-nonlinear problem. So Theorems \ref{T:fde} and \ref{T:pme} may be extended to the Finsler doubly-nonlinear problem for $p > q$ and $p < q$, respectively, in an analogous manner (however, it is not handled in the present paper for the sake of simplicity).

\subsection{Stefan problem}\label{Ss:Stefan}

We next deal with a one-phase Stefan problem, whose classical form is given as
\begin{alignat}{4}
\partial_t u(x,t) &= \Delta u(x,t) \quad &&\mbox{ for } \ x \in \Omega(t), \ t \in (0,T),\label{stefan-pde}\\
u(x,t)&=0, \quad \partial_t u(x,t) =\mu |\nabla u(x,t)|^2 \quad &&\mbox{ for } \ x \in \partial \Omega(t), \ t \in (0,T),\label{stefan-bc}\\
u(x,0)&=u_0(x) &&\mbox{ for } \ x \in \Omega(0),\label{stefan-ic}
\end{alignat}
where $\mu > 0$ is a constant and $T > 0$. Here the second relation on the boundary $\partial \Omega(t)$ is equivalent to the so-called \emph{Stefan condition} of the free boundary $\partial \Omega(t)$. Indeed, if the free boundary has a level-set representation
$$
\partial \Omega(t) = \{x \in \Rd \colon \Phi(x,t) = 0\}
$$
for some (smooth) function $\Phi : \Rd \times (0,+\infty) \to \R$, the second relation is rewritten as
$$
\partial_t \Phi(x,t) = \mu \nabla u(x,t) \cdot \nabla \Phi(x,t) \ \mbox{ for } \ x \in \partial \Omega(t), \ t \in (0,T).
$$
Higher dimensional Stefan problems are often studied in a weak formulation, although the one-dimensional case is usually studied in the classical setting. In order to introduce a weak formulation, we first extend $u$ by zero on the whole domain $\Rd$. Then the classical form \eqref{stefan-pde}--\eqref{stefan-ic} is reduced to the following weak formulation with the aid of integration by parts formula:
\begin{align*}
- \int^T_0 \int_{\Rd} \beta_0(u) \partial_t \phi \, \d x \d t
- \int_{\Rd} \beta_0(u_0(x)) \phi(x,0) \, \d x 
+ \int^T_0 \int_{\Rd} \nabla u \cdot \nabla \phi \, \d x \d t = 0
\end{align*}
for any $\phi \in C^\infty_c(\Rd \times [0,T))$. Here $\beta_0 : \R \to \R$ is given by
$$
\beta_0(u) = \begin{cases}
	      u \ &\mbox{ if } \ u > 0,\\
	      u - \mu^{-1} \ &\mbox{ if } \ u \leq 0.
	     \end{cases}
$$
Moreover, define the maximal monotone extension $\beta : \R \to 2^\R$ of $\beta_0$ by
$$
\beta(u) = \begin{cases}
	    u \ &\mbox{ if } \ u > 0,\\
	    [-\mu^{-1},0] \ &\mbox{ if } \ u = 0,\\
	    u - \mu^{-1} \ &\mbox{ if } \ u \leq 0.
	     \end{cases}
$$
Then the weak formulation implies
\begin{equation*}%\label{stefan}
\partial_t b = \Delta u, \quad b \in \beta(u) \ \mbox{ in } \Rd \times (0,T).
\end{equation*}
Moreover, a Finsler variant reads,
\begin{equation}\label{stefan}
\partial_t b = \Lap u, \quad b \in \beta(u) \ \mbox{ in } \Rd \times (0,T),
\end{equation}
which is a sort of doubly-nonlinear problem and can also be reduced to \eqref{dnp} by setting $a(x,t,\xi) = \xi$. Furthermore, we set 
$$
\hat\beta(u) = \begin{cases}
	    \frac12 u^2 \ &\mbox{ if } \ u \geq 0,\\
	    \frac12 u^2 - \mu^{-1} u \ &\mbox{ if } \ u \leq 0.
	     \end{cases}
$$
Then $\hat\beta$ is strictly convex. Thus \eqref{a-mono}--\eqref{beta-conv} are checked, and therefore, the Finsler Stefan problem \eqref{stefan} falls within the scope of the general framework presented in this paper. On the other hand, local energy estimates for \eqref{stefan} may not have yet been established.

\subsection{Extensions of the general framework}

Finally, let us also discuss possible extensions of the general framework 
for doubly-nonlinear problems established in the present paper. A first extension is concerned with the Cauchy problem for
\begin{equation}\label{dnp-g}
 \partial_t v = \mathrm{div} \, a(x,t,\nabla u) + g, \ v \in \beta(u) \ \mbox{ in } \Rd \times (0,\Sl),
\end{equation}
where $g$ is a given function satisfying
\begin{align*}
g \in \mathcal{M}(\Rd \times [0,\Sl)) \cap L^{p'}_{\rm loc}(\Rd \times (0,\Sl)) \ \mbox{ for } R > 0.
\end{align*}
Then one can take an approximate sequence $(g_n)$ in $C^\infty_c([0,\Sl] \times \Rd)$ such that
$$
g_n \to g \quad \mbox{ strongly in } L^{p'}(B_R \times (t_1,t_2))
$$
for $R > 0$ and $0 < t_1 < t_2 < \Sl$ and
$$
\int^\Sl_0 \int_{\Rd} g_n \psi \, \d x \d t \to \int^\Sl_0 \int_{\Rd} \psi \d g(x,t)
$$
for $\psi \in C^\infty_c(\Rd\times [0,\Sl))$ as $n \to +\infty$. One can extend Theorem \ref{T:loc} to the Cauchy problem for \eqref{dnp-g}. To this end, we replace \eqref{dnp:wf} in (A0) with
\begin{align}
\left\langle \partial_t v_n(t), w \right\rangle_{W^{1,p}_0(B_n)}
+ \int_{B_n} a(x,t,\nabla u_n(x,t)) \cdot \nabla w(x) \, \d x = 
 \int_{B_n} g_n(x,t) w(x) \, \d x \label{dnp-g:wf}
\end{align}
for any $w \in W^{1,p}_0(B_n)$ and a.e.~$t \in (0,\Sl)$. Then we modify local-energy estimates; in the proof of \eqref{e:un2(q-1)}, we need control an additional term as follows:
\begin{align*}
\int_{B_n} g_n(x,t) v_n(x,t)^{p'-1} \zeta_R^{p'}(x) \, \d x
&\leq \|g_n(t) \zeta_R\|_{L^{p'}(B_{2R})} \|v_n(t) \zeta_R\|_{L^{p'}(B_{2R})}^{p'-1}\\
&\leq \|v_n(t) \zeta_R\|_{L^{p'}(B_{2R})}^{p'} + C  \|g_n(t) \zeta_R\|_{L^{p'}(B_{2R})}^{p'},
\end{align*}
which leads us to derive \eqref{e:un2(q-1)} again. A similar modification should also be applied to estimate $\rho \partial_t v_n$. Furthermore, \eqref{dphi*u} will additionally involve the term
$$
\int^{\tJ}_{\tI} \int_{\BR} g_n u_n \rho \, \d x \d t \to \int^{\tJ}_{\tI} \int_{\BR} g u \rho \, \d x \d t
$$
as $n \to +\infty$. Therefore the term
\begin{equation*}
\int^{\tJ}_{\tI} \int_{\BR} g \psi \, \d x \d t
\end{equation*}
will be added to the right-hand side of the weak form \eqref{weak_form_hat}. Finally, noting that
$$
\int^t_0 \int_{\Rd} g_n \psi \, \d x \d\tau \to \int^t_0 \int_{\Rd} \psi \d g(x,\tau)
$$
for any $\psi \in C^\infty_c(\Rd \times [0,\Sl))$, we shall rewrite \eqref{wic} as
\begin{align}
- \int^t_0 \int_{\Rd} v \partial_t \psi \, \d x \d \tau
 + \int_{\Rd} v(x,t) \psi(x,t) \, \d x
  - \int_{\Rd} \psi(x,0) \, \d \mu(x)\nonumber\\
+ \int^t_0 \int_{\Rd} a(x,t,\nabla u) \cdot \nabla \psi \, \d x \d \tau = \int^t_0 \int_{\Rd} \psi(x,\tau) \d g(x,\tau)\label{eq-g}
\end{align}
for all $\psi \in C^\infty_c([0,\Sl)\times \R^\dim)$. All the other parts of the assertion remain valid.

We further consider perturbation problems for \eqref{dnp}, that is,
\begin{equation}\label{pert}
 \partial_t v = \mathrm{div} \, a(x,t,\nabla u) + f(x,t,u,\nabla u), \ v \in \beta(u) \ \mbox{ in } \Rd \times (0,\Sl),
\end{equation}
where $f : \Rd \times (0,+\infty) \times \R \times \Rd \to \R$ is a perturbation term satisfying certain assumptions, e.g., $f$ is a Carath\'eodory function, i.e., measurable in $(x,t)$ and continuous in the other variables, and $f$ complies with some growth conditions in $u$ and $|\nabla u|$. The extension to \eqref{dnp-g} may help us to extend Theorem \ref{T:loc} to the perturbed equation \eqref{pert}.

Let us finally consider structural stability of \eqref{dnp}, namely, let $(a_n(x,t,\xi))$ and $(\beta_n(u))$ be approximate sequences of $a(x,t,\xi)$ and $\beta(u)$, respectively, and replace $a(x,t,\nabla u_n)$ and $\beta(u_n)$ in (A0) and (A2) by $a_n(x,t,\nabla u_n)$ and $\beta_n(u_n)$, respectively. To this end, we may assume $G$-convergence for $a_n(x,t,\xi)$ and $\beta_n$ (cf.~\cite{Vi11}).

\appendix

\section{Some lemmas}

\begin{lemma}\label{L:meas}
Let $\mu$ be a Radon measure in $\R^\dim$. Then there exists a sequence $(\mu_n)$ in $C^\infty_c(\Rd)$ such that $\mathrm{supp}\,\mu_n \subset B_n$ and $\mu_n \to \mu$ weakly start in $\mathcal M(\Rd)$, that is,
$$
\int_{\Rd} \varphi \mu_n\,\d x \to \int_{\Rd} \varphi \, \d \mu(x) \quad \mbox{ for } \ \varphi \in C_c(\Rd)
$$
as $n\to+\infty$. Moreover, $|\mu_n| \to |\mu|$ weakly star in $\mathcal M(\Rd)$.
\end{lemma}

\begin{proof}
The Jordan decomposition theorem implies a unique decomposition of $\mu = \mu^+-\mu^-$ for some positive Radon measures $\mu^\pm$ on $\R^\dim$. Then, one can take smooth approximations $\mu^\pm_n \in C^\infty(\R^\dim)$ of $\mu^\pm$ such that $\mu^\pm_n \to \mu^\pm$ weakly star in $\mathcal M(\R^\dim)$ (e.g., by use of the Riesz representation theorem and mollifier). Now, let $\zeta_n \in C^\infty_c(\R^\dim)$ be such that $0\leq\zeta_n\leq 1$ in $\Rd$, $\mathrm{supp}\,\zeta_n \subset B_n$ and $\zeta_n \equiv 1$ on $B_{n/2}$. Let $\psi \in C_c(\R^\dim)$ and set $\varphi = \zeta_n \psi \in C_c(\R^\dim)$. It follows that
$$
\int_{\R^\dim} \psi(x) \zeta_n(x) \mu^\pm_n(x) \, \d x \stackrel{n \gg 1}= \int_{\R^\dim} \psi(x) \mu^\pm_n(x) \, \d x \to \int_{\R^\dim} \psi(x) \, \d \mu^\pm(x).
$$
Therefore setting $\mu_n := \zeta_n (\mu^+_n - \mu^-_n) \in C^\infty_c(\Rd)$, we observe that $\mu_n \to \mu$ weakly star in $\mathcal M(\Rd)$ as $n \to +\infty$. Similarly, one can also verify that $|\mu_n| \to |\mu|$ weakly star in $\mathcal M(\Rd)$.
\end{proof}

\begin{lemma}\label{a:L:rel}
Under the setting in \S \ref{Ss:wc}, the function $t \mapsto \rho v_n(t)$ is absolutely continuous with values in $(\Wpr)^*$ on $[t_1,t_2]$ such that
\begin{equation}\label{a:rel}
\rho \partial_t v_n = \partial_t [\rho v_n] \ \mbox{ in } (\Wpr)^* \ \mbox{ a.e.~in } (t_1,t_2).
\end{equation}
\end{lemma}

\begin{proof}
We observe that, for any $w \in \Wpr$ and a.e.~$t \in (t_1,t_2)$,
\begin{align*}
%\lefteqn{
 \left\langle \rho \partial_t v_n(t), w \right\rangle_{\Wpr}
%}\\
&\stackrel{\text{def.}}= \left\langle \partial_t v_n(t) , \overline{\rho w} \right\rangle_{W^{1,p}_0(B_n)}\\
&= \lim_{h \to 0} \left\langle \frac{v_n(t+h)-v_n(t)}h, \overline{\rho w} \right\rangle_{W^{1,p}_0(B_n)}\\
&= \lim_{h \to 0} \int_{\BR} \frac{v_n(t+h)-v_n(t)}h \rho w \, \d x\\
&= \lim_{h \to 0} \int_{\BR} \frac{\rho v_n(t+h)-\rho v_n(t)}h  w \, \d x\\
&= \lim_{h \to 0} \left\langle \frac{\rho v_n(t+h)-\rho v_n(t)}h, w \right\rangle_{\Wpr},
\end{align*}
whence follows that $t \mapsto \rho v_n(t)$ is weakly differentiable in $\Wpr$ at a.e.~$t \in (t_1,t_2)$, and moreover, the weak derivative coincides with $\rho \partial_t v(t)$ in $(\Wpr)^*$. On the other hand, we observe that
\begin{align*}
\lefteqn{
 \left| \left\langle \rho v_n(t)-\rho v_n(s), w \right\rangle_{\Wpr}\right|
}\\
&= \left| \left\langle v_n(t)- v_n(s), \overline{\rho w} \right\rangle_{W^{1,p}_0(B_n)}\right|\\
&\leq \left\| v_n(t)-v_n(s)\right\|_{W^{-1,p'}(B_n)} \|\overline{\rho w}\|_{W^{1,p}_0(B_n)}\\
&\lesssim \|w\|_{\Wpr} \int^t_s \|\partial_\sigma v_n(\sigma)\|_{W^{-1,p'}(B_n)} \, \d \sigma
\end{align*}
for $w \in \Wpr$ and $0 < s < t < \Sl$ (indeed, $v_n \in W^{1,p'}(0,\Sl;W^{-1,p'}(B_n))$ by (A0)). Therefore $t \mapsto \rho v_n(t)$ is absolutely continuous in $(\Wpr)^*$ on $[t_1,t_2]$, and hence, it is strongly differentiable in $(\Wpr)^*$ a.e.~in $(t_1,t_2)$. Since the weak derivative coincides with the strong one, we obtain the relation \eqref{a:rel}.
\end{proof}

\begin{lemma}[Comparison principle for integral inequalities]\label{A:L:cp}
Let $f : \R \to \R$ be a non-decreasing function and let $a_+, a_-$ be constants such that $a_- < a_+$. Let $\phi_+, \phi_-: [0,+\infty) \to \R$ and $k : [0,+\infty) \to \R_+$ be {\rm (}possibly \emph{discontinuous}{\rm )} functions such that $t \mapsto k(t) f(\phi_\pm (t)) \in L^1_{\rm loc}([0,+\infty))$ and
\begin{align*}
\phi_-(t) &\leq a_- + \int^t_0 k(\tau) f(\phi_-(\tau)) \, \d \tau,\\
\phi_+(t) &\geq a_+ + \int^t_0 k(\tau) f(\phi_+(\tau)) \, \d \tau
\end{align*}
for $t > 0$. Then $\phi_-(t) < \phi_+(t)$ for all $t > 0$. 
\end{lemma}

\begin{proof}
We first note that
\begin{equation}\label{ap:subtr}
\phi_-(t) - \phi_+(t)
\leq \underbrace{a_- - a_+}_{ < 0} + \int^t_0 \underbrace{k(\tau) \left[f(\phi_-(\tau)) - f(\phi_+(\tau)) \right]}_{\in \ L^1_{\rm loc}([0,+\infty))} \, \d \tau,
\end{equation}
whence follows that one can take $t_* > 0$ such that
$$
\phi_-(t) - \phi_+(t) < 0 \quad \mbox{ for } \ t \in [0,t_*].
$$
Assume on the contrary that there exists a finite $t_1 > 0$ such that
$$
\phi_-(t) - \phi_+(t) < 0 \ \mbox{ for } \ t \in (0,t_1), \quad \phi_-(t_1) - \phi_+(t_1) \geq 0.
$$
Recall \eqref{ap:subtr} and substitute $t = t_1$. It then follows that
\begin{align*}
0 &\leq \phi_-(t_1) - \phi_+(t_1)\\
&\leq a_- - a_+ + \int^{t_1}_0 k(\tau) \left[ f(\phi_-(\tau)) - f(\phi_+(\tau)) \right] \, \d \tau < 0,
\end{align*}
which implies a contradiction. Thus we conclude that $t_1 = +\infty$.
\end{proof}

\pagestyle{myheadings}

\end{document}